\documentclass[11pt]{amsart}
\usepackage{amsmath,amssymb,amsthm,amsfonts,setspace,hyperref,dsfont,yhmath}

\newcommand{\NN}{\mathbb{N}}
\newcommand{\RR}{\mathbb{R}}

\newcommand{\CC}{\mathbb{C}}

\newcommand{\QQ}{\mathbb{Q}}

\newcommand{\TT}{\mathbb{T}}
\newcommand{\ZZ}{\mathbb{Z}}

\newcommand{\norm}[1]{\lVert#1\rVert}
\newcommand{\abs}[1]{\lvert#1\rvert}

\newtheorem{theorem}{Theorem}[section]
\newtheorem{corollary}[theorem]{Corollary}

\newtheorem{lemma}[theorem]{Lemma}
\newtheorem{proposition}[theorem]{Proposition}

\theoremstyle{definition}
\newtheorem{definition}[theorem]{Definition}
\newtheorem{remark}[theorem]{Remark}

\numberwithin{equation}{section}

\begin{document}

\title{Local rigidity of higher rank non-abelian action on tours}
\footnotetext{The author is supported by NSF grants DMS-1346876.}
\author[]{ Zhenqi Jenny Wang }
\address{Department of Mathematics\\
        Michigan State University\\
        East Lansing, MI 48824,   USA}
\email{wangzq@math.msu.edu
}

\begin{abstract}
In this paper, we show local smooth  rigidity for higher rank ergodic nilpotent action by toral automorphisms and prove the existence of such action on any torus $\TT^N$ for any even $N\geq 6$. We also give examples of smooth rigidity of actions having rank-one factors. The method is a generalization of the KAM (Kolmogorov-Arnold-Moser) iterative scheme.

\end{abstract}
\maketitle
\section{Introduction and main result}
Let $\mathcal{H}$ be a finitely generated group and $\alpha:\mathcal{H}\rightarrow GL(N,\ZZ)$ be a homomorphism, where $GL(N,\ZZ)$ is the group of integer $N\times N$ matrices with determinant $\pm 1$. Then $\alpha$ induces a natural action on $\TT^N$ by automorphism. We say that action $\alpha$ is a \emph{higher rank ergodic action} if $\alpha(\mathcal{H})$ contains two ergodic elements $A,\,B$ such that $A^{k_1}B^{k_2}$ is ergodic if $k=(k_1,k_2)\in\ZZ^2\backslash 0$.
\begin{definition} The action $\alpha$ of  $\mathcal{H}$ on $\TT^N$ is $C^{k,r,\ell}$ {\em locally
rigid} if any $C^k$ perturbation $\tilde{\alpha}$ which is
sufficiently $C^r$ close to $\alpha$ on a compact generating
set is $C^\ell$ conjugate to $\alpha$.
\end{definition}
In contrast to  the structural stability ($C^0$ rigidity) of diffeomorphisms and flows in hyperbolic dynamics,
where differentiable rigidity is mostly like impossible, the presence of a large group action frequently allows one to improve
the regularity of the conjugacy.  The study of local differentiable rigidity of group actions has had two primary progresses, one from higher rank abelian action, see
\cite{Damjanovic4}, \cite{Damjanovic3}, \cite{fisher}, \cite{Spatzier1}, \cite{Spatzier2}, \cite{Spatzier}; the other from lattice in a semisimple Lie group, see \cite{fisher1}.

The most general condition in the setting of $\ZZ^k\times\RR^l$ $k+l\geq2$, actions, which
leads to various rigidity phenomena (cocycle rigidity, local differentiable rigidity,
measure rigidity, etc.), is the following:

\noindent $(\mathfrak{R})$ the group $\ZZ^k\times\RR^l$ contains a subgroup $S$ isomorphic to $\ZZ^2$ such that every element other than identity
acts ergodically with respect to the standard invariant measure obtained
from Haar measure.

After noticing this condition, one can pass to the $\ZZ^2$ sub-action to establish the smooth conjugacy and then show that the conjugacy obtained also conjugates the other elements in $\ZZ^k\times\RR^l$.

A lattice in a connected semisimple Lie group $G$ of non-compact type is ``large" in the following sense: its Zariski closure is $G$. The rigidity theorem of lattice,
such as  Margulis' super-rigidity theorem and Zimmer's cocycle super-rigidity theorem play important roles in reduction of lattice actions.

If $\mathcal{H}$ is nilpotent, as we showed in Proposition \ref{po:1} all elements in $\alpha([\mathcal{H},\mathcal{H}])$ are not ergodic, where $[\mathcal{H},\mathcal{H}]$ is the commutator group  of $\mathcal{H}$, which means we can't pass to a $\ZZ^2$ ergodic action; and since any representation of an semisimple Lie group preserves semisimple and unipotent elements respectively, it is impossible to extend the action $\alpha$ to $G$ even if $\mathcal{H}$ sits inside a semisimple group $G$.

In this paper we prove local differentiable rigidity for higher rank ergodic nilpotent action by toral automorphisms. We also show the existence of genuine partially hyperbolic nilpotent action. Our method is  based on KAM-type iteration scheme
that was first introduced in \cite{Damjanovic4} and was later developed in \cite{Damjanovic3}. In their proofs the
commutativity of the action is essential. The natural difficulty in
non-abelian type arguments is related to the complexity of the cocycle equations related to commutator relations between non-abelian generators. For example, in \cite{wang} we extended the method to treat the nipotent action of length $2$, that is, the Heisenberg group action. However, even for the most simple non-abelian case, the calculation is complex. To prove the theorem, we make sufficient reduction and establish new orbit increasing relations between two ergodic generators.
\begin{definition} An action $\alpha'$ of $\mathcal{H}$ on $\TT^{N'}$ is an algebraic factor of an action $\alpha$ of $\mathcal{H}$ on $\TT^{N}$ if there exists an
epimorphism $h:\TT^N\rightarrow\TT^{N'}$ such that $h\circ\alpha=\alpha'\circ h$.

An action $\alpha'$ is a rank one factor if it is an algebraic factor and if $\alpha'(\mathcal{H})$
contains a cyclic subgroup of finite index.
\end{definition}
Condition $(\mathfrak{R})$ is always viewed as a paradigm for differential rigidity phenomena. An ergodic action $\alpha$ by toral automorphisms has no nontrivial rank one factors if and only if it satisfies condition $(\mathfrak{R})$ (see for example \cite{Starkov}). All the examples given so far are based on this condition. In this paper we obtain a class of examples where the above condition fails but enjoys differential rigidity property. The basic idea is for dual orbits with large projections in the subspace admitting rank-one factors, its increasing speed for well chosen non-ergodic (or unipotent) elements is fast enough to obtain tame estimates for the size of obstructions; and we have enough such non-ergodic (or unipotent) elements to cover all integer vectors inside the subspace. The polynomial increasing speed for unipotent elements play a crucial role in the proof.

\subsection{Statement of the main result}
Let $\mathcal{H}$ be a finitely generated group and $\alpha:\mathcal{H}\times\TT^N\rightarrow\TT^N$ is given by an embedding
$\rho_\alpha:\ZZ^N\rightarrow GL(N,\ZZ)$ so that
\begin{align*}
  \alpha(g,x)=\rho_\alpha(g)(x)
\end{align*}
for any $g\in\mathcal{H}$ and any $x\in\TT^N$. Let $\alpha:\mathcal{H}\times\TT^N\rightarrow\TT^N$ be an action of $\mathcal{H}$ by automorphisms of the $N$-dimensional torus. Throughout the paper, we will write simply $\alpha(g)$ for $\bar{g}$ if needed.

In next theorem we assume $\mathcal{H}$ is not abelian, otherwise it is Theorem $1$ in \cite{Damjanovic4}.
\begin{theorem}\label{th:1}
 If $\mathcal{H}$ is nilpotent (not abelian) and $\alpha$ is a higher rank ergodic action, then there exists a constant $l(\alpha,N)$ such that $\alpha$ is $C^{\infty,l,\infty}$ locally rigid.
\end{theorem}
 For any unipotent element  $U\in GL(N,\ZZ)$  let $p_1(U)=\{v\in\RR^N:Uv=v\}$. Suppose $A_1$ is an ergodic element in $GL(N,\ZZ)$ and $A_i$, $1\leq i\leq n$ are unipotent element  $U\in GL(N,\ZZ)$ satisfying $A_1A_i=A_iA_1$ and $(A_i-I_N)^2=0$, $2\leq i\leq n$. Also suppose $\bigcap_{i=2}^np_1(A_i)=\{0\}$. Let $\mathcal{H}$ be a group generated by $A_i$, $1\leq i\leq n$.
\begin{theorem}\label{th:2}

For the action $\alpha$ described above, there exists a constant $l(\alpha,N)$ such that $\alpha$ is $C^{\infty,l,\infty}$ locally rigid.
\end{theorem}
\begin{remark}
\end{remark}

Let us call an action of $\mathcal{H}$ by automorphisms of a torus
genuinely partially hyperbolic if it is ergodic with respect to Lebesgue measure
but no element of the action is hyperbolic (Anosov). It is easy to see that this is
equivalent to simultaneous existence of
\begin{enumerate}
  \item an element of the action none of whose eigenvalues is a root of unity and

  \smallskip
  \item an invariant linear foliation on which there is no exponential expansion/contraction
for any element of the action.
\end{enumerate}
As before, such an action is higher rank if and only if it contains two elements $A,\,B$ such that $A^{k_1}B^{k_2}$ is ergodic if $k=(k_1,k_2)\in\ZZ^2\backslash 0$.
\begin{theorem}\label{th:3}
Genuinely partially hyperbolic higher rank nilpotent $\mathcal{H}$ actions exist: on
any torus of even dimension $N\geq 6$ there are irreducible examples while on any
torus of odd dimension $N\geq9$ there are only reducible examples. There are no
examples on tori of dimension $N\leq5$ and $N=7$.
\end{theorem}

\section{Setting of the problem and the KAM scheme} \label{sec:3}Before proceeding to specifics we will show how the general KAM scheme described in \cite[Section 3.3]{Damjanovic4} and \cite[Section 1.1]{Damjanovic3}
is adapted to the $\mathcal{H}$ action $\alpha$, which also clarifies the proof line of the paper.

\smallskip
\noindent \textbf{Step 1}. \emph{Setting up the linearized equation}

\smallskip
Let $\widetilde{\alpha}$ be a small perturbation of $\alpha$. To prove the existence
of a $C^\infty$ map $H$ such that $\widetilde{\alpha}\circ H=H\circ \alpha$, we need to solve the nonlinear conjugacy problem
\begin{align*}
  \alpha(g)\circ\Omega-\Omega\circ\alpha(g)=-R_g\circ(I+\Omega), \qquad \forall g\in\mathcal{H}
\end{align*}
where $\widetilde{\alpha}(g,\cdot)=\alpha(g,\cdot)+R_g(\cdot)$ and $H=I+\Omega$; and the corresponding linearized  conjugacy equation is
\begin{align}\label{for:10}
  \alpha(g)\circ\Omega-\Omega\circ\alpha(g)=-R_g, \qquad \forall g\in\mathcal{H}
\end{align}
for small $\Omega$ and $R$.

Lemma \ref{le:2} shows that obtaining a $C^\infty$ conjugacy
for one ergodic generator suffices for the proof of Theorem \ref{th:1}. Hence we just need to solve equation
\eqref{for:10} for one ergodic generator.

\smallskip
\noindent \textbf{Step 2}. \emph{Solving the linearized conjugacy equation for a particular
element.}

\smallskip
We classify the obstructions for solving the linearized equation \eqref{for:10} for an individual generator (see Lemma \ref{le:8} and \ref{le:5})
and obtain tame estimates are obtained for the solution. This means finite loss
of regularity in the chosen collection of norms in the Fr$\acute{e}$chet spaces,
such as $C^r$ or Sobolev norms.

\smallskip
\noindent \textbf{Step 3}. \emph{Constructing projection of the perturbation to the twisted cocycle space.}

 Since $R_g(x)$, where $g\in\mathcal{H}$ and $x\in\TT^N$ is a twisted cocycle not over $\alpha$ but over $\widetilde{\alpha}$ (see Lemma 3.3 of \cite{Damjanovic4}), \eqref{for:10} is not a twisted coboundary
equation over the linear action $\alpha$, just an approximation. Then we define the cocycle difference function:
\begin{align}\label{for:57}
  \mathcal{L}(x,y)\stackrel{\rm def}{=}R_x\circ \bar{y}+\bar{x}R_y-R_y\circ \overline{xz}-\bar{y}R_x\circ \bar{z}-\overline{yx}R_z
\end{align}
for $x,y\in \mathcal{H}$ and $z=x^{-1}y^{-1}xy$.

It is clear that if $\mathcal{L}=0$ then $R$ is a twisted cocycle over $\alpha$.  But even if \eqref{for:10} is a twisted coboundary over $\alpha$, it is impossible to produce a $C^\infty$ conjugacy for a single ergodic generator of the action.

Therefore, when $\mathcal{H}$ is nilpotent (not abelian), i.e., to prove Theorem \ref{th:1} we consider $\mathfrak{n}+2$ generators, $g_1$, $d_0$ and $d_j=D_j(g_1,d_0)$, $0\leq j\leq \mathfrak{n}$ ($\mathfrak{n}$ and $D_j$ are defined in \eqref{for:103} of Section \ref{sec:1})
and reduce the problem of solving the linearized equation \eqref{for:10} to solving
simultaneously the following system:
\begin{align}\label{for:11}
  \mathcal{A}\circ\Omega-\Omega\circ \mathcal{A}&=-R_{g_1}\notag\\
  \overline{d_j}\circ\Omega-\Omega\circ \overline{d_j}&=-R_{d_j}
  \end{align}
where $\mathcal{A}=\overline{g_1}$ and $\overline{d_0}=\mathcal{B}$. $\mathcal{A},\,\mathcal{B}$ are ergodic generators constructed at the beginning of Section \ref{sec:2}).

When $\mathcal{H}=\ZZ^2$, to prove Theorem \ref{th:2} we consider $2$ generators $A_1$ and $A_2$
and reduce the problem of solving the linearized equation \eqref{for:10} to solving
simultaneously the following system:
\begin{align}
  A_i\circ\Omega-\Omega\circ A_i&=-R_{A_i},\qquad i=1,\,2.
  \end{align}
As mentioned above, $R$ does not satisfy this twisted cocycle condition:
\begin{align*}
  \mathcal{L}(x,y)=R_x\circ \bar{y}+\bar{x}R_y-R_y\circ \overline{xz}-\bar{y}R_x\circ \bar{z}-\overline{yx}R_z=0.
\end{align*}
However the difference
\begin{align*}
 \mathcal{L}(g_1,d_j), \qquad 0\leq j\leq \mathfrak{n}
\end{align*}
when $\mathcal{H}$ is nilpotent (not abelian) or
\begin{align*}
 \mathcal{L}(A_1,A_2)
\end{align*}
when $\mathcal{H}=\ZZ^2$ are quadratically small with respect to $R$,  (see Lemma \ref{le:3} and Remark \ref{re:2}). More precisely, the perturbation $R$ can be split into two terms
\begin{align*}
 R=\mathcal{P}R+\mathcal{E}(R)
\end{align*}
so that $\mathcal{P}R$ is in the space of twisted cocycles and the error $\mathcal{E}(R)$ is bounded by the size of $\mathcal{L}$ with the fixed loss of regularity.
More precisely, the system
\begin{align}\label{for:104}
  -\mathcal{P}R_{g_1}&=-(R_{g_1}-\mathcal{E}(R_{g_1}))=\mathcal{A}\Omega-\Omega\circ \mathcal{A},\notag\\
   -\mathcal{P}R_{d_j}&=-(R_{d_j}-\mathcal{E}(R_{d_j}))=\overline{d_j}\Omega-\Omega\circ \overline{d_j}, \quad 0\leq j\leq \mathfrak{n}
   \end{align}
when $\mathcal{H}$ is nilpotent (not abelian); or
 \begin{align}\label{for:37}
  -\mathcal{P}R_{A_i}&=-(R_{A_i}-\mathcal{E}(R_{A_i}))=A_i\Omega-\Omega\circ A_i,\qquad 1\leq i\leq 2
   \end{align}
when $\mathcal{H}=\ZZ^2$, have a common solution $\Omega$ after subtracting a part bounded by the size of $\mathcal{L}$, which is quadratically small to $R$.
(see Proposition \ref{po:3} and \ref{po:4}).

\smallskip
\noindent \textbf{Step 4}. \emph{Conjugacy transforms the perturbed action into an action
quadratically close to the target.}

The common approximate solution $\Omega$ to the equations \eqref{for:104} above provides a new perturbation
\begin{align*}
\widetilde{\alpha}^{(1)}\stackrel{\rm def}{=}H^{-1}\circ\widetilde{\alpha}\circ H
\end{align*}
where $H=I+\Omega$, is much closer to $\alpha$ than $\widetilde{\alpha}$; i.e., the new error
\begin{align*}
 R^{(1)}\stackrel{\rm def}{=}\widetilde{\alpha}^{(1)}-\alpha
\end{align*}
is expected to be small with respect to the old error $R$.

\smallskip
\noindent \textbf{Step 5}. \emph{The process is iterated and the conjugacy is obtained.}

\smallskip
The iteration process is set and is carried out, producing a $C^\infty$ conjugacy which works for the action
generated by the $\mathfrak{n}+1$ generators $\mathcal{A}$, $\overline{d_0}=\mathcal{B}$ and $\overline{d_j}$, $0\leq j\leq \mathfrak{n}$ when $\mathcal{H}$ is nilpotent (not abelian);  and works for the action
generated by the $2$ generators $A_1$ and $A_2$ when $\mathcal{H}=\ZZ^2$. Ergodicity assures that it works
for all the other elements of the action $\alpha$.

\smallskip
What is described above highlights the essential features of the KAM scheme for the $\mathcal{H}$ action on torus.
The last two steps can follow
Section 5.2-5.4 in \cite{Damjanovic4} word by word with minor modification. Hence completeness of Step 2 and 3
admits the conclusion of Theorem \ref{th:1} an \ref{th:2}.

At the end of the this seciton, we prove a simple lemma which shows that obtaining a $C^\infty$ conjugacy
for one ergodic generator suffices for the proof of Theorem \ref{th:1} and \ref{th:2}. Next, we state a fact which is necessary for the proof.
\begin{lemma} \cite[Lemma 3.2]{Damjanovic4}\label{le:3}
For any $C^1$ small enough map $F:\TT^N\rightarrow\TT^N$, if $AF=F\circ A$, where $A\in GL(N,\ZZ)$ and is ergodic, then $F=0$.
\end{lemma}

\begin{lemma}\label{le:2}
Let $\alpha$ be a finitely generated nilpotent group $\mathcal{H}$ action by automorphisms of $\TT^N$ such that for
some $g\in \mathcal{H}$ the automorphism $\alpha(g)$ is ergodic. Let $\widetilde{\alpha}$ be a $C^1$ small perturbation
of $\alpha$ such that there exists a $C^\infty$ map $H:\TT^N\rightarrow\TT^N$ which is $C^1$ close to identity
and satisfies
\begin{align*}
 \widetilde{\alpha}(g)\circ H=H\circ\alpha(g).
\end{align*}
Then $H$ conjugates the corresponding maps for all the other elements of the action;
i.e., for all $h\in \mathcal{H}$ we have
\begin{align}\label{for:23}
 \widetilde{\alpha}(h)\circ H=H\circ\alpha(h).
\end{align}
\end{lemma}
\begin{proof}
Suppose $\mathcal{H}$ has a lower central series of length $n$, i.e., a sequence of subgroups
\begin{align*}
\{e\}=\mathcal{H}_0\lhd \mathcal{H}_1\lhd\cdots\lhd \mathcal{H}_n=\mathcal{H}
\end{align*}
such that $[\mathcal{H}, \mathcal{H}_{j+1}]=\mathcal{H}_j$ where $[\mathcal{H}, \mathcal{H}_{j+1}]$ denotes the commutator of $\mathcal{H}$ and $\mathcal{H}_{j+1}$. Fix a set of generators for each $\mathcal{H}_j$ and denote this set by $S_j$. Note that each $S_j$ can be chosen to be finite (see Lemma $2$ of \cite{brown}).

We will use induction to show $H$ conjugates all the other elements of $\alpha$ and $\widetilde{\alpha}$.
Let $h$ be any element in $S_1$. Since $hg=gh$ it follows from \eqref{for:23} and
commutativity that
\begin{align*}
 &\alpha(g)\circ \tilde{h}=\tilde{h}\circ\alpha(g)
 \end{align*}
where $\tilde{h}=\alpha(h)\circ\mathcal{H}^{-1}\circ\widetilde{\alpha}(h)^{-1}\circ\mathcal{H}$. Lemma \ref{le:3} shows that $\tilde{h}=I$, which means
\begin{align*}
 \widetilde{\alpha}(h)\circ H=H\circ\alpha(h).
\end{align*}
Arbitrariness of $h$ implies that $H$ conjugates all the elements in $A_1$ of $\alpha$ and $\widetilde{\alpha}$

Suppose $H$ conjugates all the elements in $\mathcal{H}_j$ of $\alpha$ and $\widetilde{\alpha}$. For any $h\in S_{j+1}$ since $[\mathcal{H}, \mathcal{H}_{j+1}]\leq \mathcal{H}_j$, there exits $h_1\in \mathcal{H}_j$ such that $hg=ghh_1$. By assumption, we have
\begin{align*}
 \widetilde{\alpha}(gh_1^{-1})\circ H=H\circ\alpha(gh_1^{-1}).
\end{align*}
We obtain
\begin{align*}
 &\alpha(g)\circ (\alpha(h)\circ H^{-1}\circ\widetilde{\alpha}(h)^{-1}\circ H)\\
 &=\alpha(h)\circ (\alpha(gh_1^{-1})\circ H^{-1})\circ\widetilde{\alpha}(h)^{-1}\circ H\\
 &=\alpha(h)\circ H^{-1}(\widetilde{\alpha}(gh_1^{-1})\circ\widetilde{\alpha}(h)^{-1})\circ\mathcal{H}\\
 &=\alpha(h)\circ H^{-1}\widetilde{\alpha}(h^{-1})\circ(\widetilde{\alpha}(g)\circ\mathcal{H})\\
 &=(\alpha(h)\circ H^{-1}\widetilde{\alpha}(h^{-1})\circ\mathcal{H})\circ\alpha(g).
 \end{align*}
By using Lemma \ref{le:3} again, we get $\alpha(h)\circ H^{-1}\widetilde{\alpha}(h^{-1})\circ H=I$, which implies that $H$ conjugates $\alpha(h)$ and $\widetilde{\alpha}(h)$.  Since $S_{j+1}$ is a set of germinators of $\mathcal{H}_{j+1}$, $H$ conjugates all the elements in $\mathcal{H}_{j+1}$ of $\alpha$ and $\widetilde{\alpha}$. Hence we get the conclusion.
\end{proof}

\subsection{Basic facts and some notations}

\subsection{Some notations} \label{sec:1}
We try as much as possible to develop a unified system of notations. We will use notations from this section throughout   subsequent sections. So the reader
should consult this section  if an unfamiliar symbol appears.
\begin{enumerate}
\item \label{for:103} Suppose $\mathcal{H}$ is nilpotent. If $\mathcal{H}$ has a lower central series of length $\mathfrak{n}$, i.e., a sequence of subgroups
\begin{align*}
\{e\}=\mathcal{H}_0\lhd \mathcal{H}_1\lhd\cdots\lhd \mathcal{H}_\mathfrak{n}=\mathcal{H}
\end{align*}
such that $[\mathcal{H}, \mathcal{H}_{j+1}]=\mathcal{H}_j$ where $[\mathcal{H}, \mathcal{H}_{j+1}]$ denotes the commutator group of $\mathcal{H}$ and $\mathcal{H}_{j+1}$.

For any two elements $x,\,y\in\mathcal{H}$, define $\mathfrak{n}-1$ elements in $\mathcal{H}$ as follows:
$D_1(x,y)=x^{-1}y^{-1}xy$, $D_{i+1}(x,y)=x^{-1}D_{i}(x,y)^{-1}xD_i(x,y)$, $1\leq i\leq \mathfrak{n}-2$.

\smallskip

\item \label{for:12}For any $F\in GL(N,\ZZ)$, $\norm{F}:=\sup\{\norm{Fv}:v\in\RR^N\text{ with }\norm{v}=1\}$ and $\norm{F}_{\text{min}}:=\min\{\norm{Fv}:v\in\RR^N\text{ with }\norm{v}=1\}$. Then $\norm{F}_{\text{min}}=\norm{F^{-1}}^{-1}$. For any $m$-Jordan block $J$ of $F$ with eigenvalue $\lambda$, we have
    \begin{align}\label{for:39}
     \norm{F^n\mid_{J}}\leq C\abs{\lambda}^n(\abs{n}+1)^m,\qquad \forall n\in\ZZ.
    \end{align}
For a sequence of matrices $F_i\in GL(N,\ZZ)$ $\prod_{i=1}^nF_i\overset{\text{def}}{=}F_1\cdots F_n$.
\smallskip
  \item Let $F\in GL(N,\ZZ)$ be an ergodic integer matrix. The dual map $F^*$ on $\ZZ^N$ induces a decomposition of $\RR^N$ into expanding,
neutral and contracting subspaces. We will denote the expanding subspace by $V_1(F)$, the contracting subspace by $V_3(F)$ and the neutral subspace by $V_2(F)$.
\begin{align*}
 \RR^N_F=V_1(F)\bigoplus V_2(F) \bigoplus V_3(F).
\end{align*}
All three subspaces $V_i(F)$, $i=1$, $2$, $3$ are $F$ invariant and
\begin{align}
  \norm{F^iv}&\geq C\rho^i\norm{v}, &\rho>1,\quad &i\geq 0,\quad &v\in V_1(F), \notag\\
  \norm{F^iv}&\geq C\rho^{-i}\norm{v}, &\rho>1,\quad &i\leq 0,\quad &v\in V_3(F), \notag\\
  \norm{F^iv}&\geq C\abs{i}^{-N}\norm{v}, &\rho>1,\quad &i\neq 0,\quad &v\in V_2(F)\label{for:4}.
\end{align}
Here $C$ is a constant dependent on $F$.

\smallskip
  \item For $v\in\ZZ^N$, $\abs{v}\overset{\text{def}}{=}\max\{\norm{\pi_1(v)},\,\norm{\pi_2(v)},\,\norm{\pi_3(v)}\}$ where $\norm{\cdot}$ is Euclidean
norm and $\pi_i(v)$ are projections of $v$ to subspaces $V_i$ ($i=1,\,2,\,3$) from \eqref{for:4},
that is, to the expanding, neutral, and contracting subspaces of $\RR^N$
for $F$; we will use the norm which is more convenient in a particular situation;
those are equivalent norms, the choice does not affect any results).

\smallskip
  \item \label{for:111}For $v\in\ZZ^N$ we say $v$ is \emph{mostly} in $i(F)$ for $i=1,\,2,\,3$ and will write $v\hookrightarrow i(F)$, if
the projection $\pi_i(v)$ of $v$ to the subspace $V_i$ corresponding to $F$ is sufficiently
large:
\begin{align*}
 \abs{v}=\norm{\pi_i(v)};
\end{align*}
if furthermore,
\begin{align*}
 \abs{v}=\norm{\pi(v)}>\norm{\pi_j(v)}, \qquad j\neq i
\end{align*}
then we say that $v$ is \emph{absolutely} in $i(F)$ and write $v\rightarrow i(F)$. The notation $v\hookrightarrow 1,2(F)$  will be used for $v$ which is mostly in $1(F)$ or mostly
in $2(F)$. The notation $v\rightarrow 1,2(F)$ is defined accordingly.

\smallskip
\item \label{for:105} Call $n\in\ZZ^N$ minimal
and denote it by $\mathcal{M}_F(n)$ if $v$ is the lowest point on its $F$ orbit in the sense that $n\hookrightarrow 3(F)$
and $Fn\rightarrow 1,2(F)$. We can assume there is one such
minimal point on each nontrivial dual $F$ orbit (other wise we consider $F^n$ where $n$ is big enough), we choose one on each dual $F$ orbit
and denote it by $\mathcal{M}_F(n)$. Then $\mathcal{M}_F(n)$ is substantially large both in $1,2(F)$ and in $3(F)$. Set $E_F=\{\mathcal{M}_F(v):v\in\ZZ^n\backslash 0\}$.

\smallskip

\item \label{for:44}Let $A$ and $B$ be the two ergodic generators for $\alpha$ when $\mathcal{H}$ is nilpotent. In what follows, $C$ will denote any constant that depends only on the given
linear $\mathcal{H}$ action $\alpha$ and on the dimension of the
torus. $C_{x,y,z,\cdots}$ will denote any constant that in addition to the above depends
also on parameters $x$, $y$, $z$, $\cdots$.

\smallskip
\item \label{for:110}Let $\theta$ be a $C^\infty$ function. Then we can write $\theta=\sum_{n\in\ZZ^N}\widehat{\theta}_ne_n$
where $e_v=e^{2\pi \sqrt{-1}v\cdot x}$ are the characters. Then
\begin{itemize}
  \item [(i)] $\norm{\theta}_a\overset{\text{def}}{=}\sup_v\abs{\widehat{\theta}_v}\abs{v}^a$, $a>0$.

  \smallskip

  \item [(ii)] The following relations hold (see, for example, Section 3.1 of \cite{LLAVE}):
  \begin{align*}
   \norm{\theta}_r\leq C\norm{\theta}_{C^r},\qquad \norm{\theta}_{C^r}\leq C\norm{\theta}_{r+\sigma}
  \end{align*}
  where $\sigma>N+1$, and $r\in\NN$.

  \smallskip

  \item [(iii)] For any $F\in SL(N,\ZZ)$ $(\widehat{\theta\circ F})_n=\widehat{\theta}_{(F^\tau)^{-1}n}$ where $F^\tau$ denotes transpose matrix.
  We call $(F^\tau)^{-1}$ the dual map on $\ZZ^N$. To simplify the notation in the rest of the paper, whenever there is no
confusion as to which map we refer to we will denote the dual map by the same
symbol $F$.
\end{itemize}

\smallskip
\item \label{for:46}For a map $\mathcal{F}$ with coordinate functions $f_i$ ($i=1,\cdots,k$) define $\norm{\mathcal{F}}_a\overset{\text{def}}{=}\max_{1\leq i\leq k}\norm{f_i}_a$. For two maps $\mathcal{F}$ and $\mathcal{G}$ define $\norm{\mathcal{F},\mathcal{G}}_a\overset{\text{def}}{=}\{\norm{\mathcal{F}}_a,\norm{\mathcal{G}}_a\}$.
$\norm{\mathcal{F}}_{C^r}$ and $\norm{\mathcal{F},\mathcal{G}}_{C^r}$ are defined similarly. For any $v\in \ZZ^N$
$\widehat{\mathcal{F}}_v\overset{\text{def}}{=}
((\widehat{f_1})_v,\cdots,(\widehat{f_k})_v)$. For any $F\in GL(N,\ZZ)$,
\begin{align*}
 \Delta_F\mathcal{F}\overset{\text{def}}{=}F\mathcal{F}-\mathcal{F}\circ F.
\end{align*}

\end{enumerate}
\subsection{Basic facts about nilpotent actions on torus}
For an abelian action over a compact manifold, there
is a splitting of the tangent bundle into Lyapunov spaces (see \cite{Spatzier}). Proposition \ref{po:1} shows that similar result holds for nilpotent action by toral automorphism. In fact, it is not hard to show it also applies to general nilpotent actions for any length $\mathfrak{n}$. The case of $\mathfrak{n}=2$ was prove in \cite{wang}.

The next two lemmas are essential for the proof of the proposition. The first lemma shows that for any products with elements coming from a finite set, we can reorder the the product with a tame price: the word growth rate is polynomial; furthermore, if the size of these elements increase tamely, then the size of the product also has tame increasing rate.
\begin{lemma}\label{le:6}
Let $S$ be a finite set in $\mathcal{H}$ and set $S'=\{[s^{\delta_i}_i,\cdots[s_1^{\delta_1},s_2^{\delta_2}]\cdots]:s_i\in S,\,\delta_i=\pm1\}$. Then there exists a polynomial $p$, such that any product $\prod_{i=1}^ns_i$ where $s_i\in S$ can be expressed as
\begin{align*}
 \prod_{i=1}^ns_i=d(s_1')^{k_1}\cdots(s_j')^{k_j}, \qquad j\leq \sharp(S),\quad \sum_{i=1}^jk_i=n
\end{align*}
where $s_i'\in \{s_i:1\leq i\leq n\}$, and $(s_1')^{k_1}\cdots(s_j')^{k_j}$ is a reordered product of $\prod_{i=1}^ns_i$ and $d$ is a product of elements in $S'$ with word length bounded by $p(n)$.
\end{lemma}
\begin{proof}
We prove the following claim $(*)$ instead, which implies the conclusion immediately.

\noindent $(*)$ There exists a polynomial $p$, such that any product $(\prod_{i=1}^ns_i)(s_n\prod_{i=1}^{n-1} s_i)^{-1}$ where $s_i\in S$ can be expressed as s product of elements in $S'$ with word length bounded by $p(n)$.

It is clear that $S'$ is also a finite set. Let $S'_i$ denote the set of elements in $S'$ with (commutator) length $i$. Then $S'=\bigcup_{i=2}^\mathfrak{n} S'_i$. Moving $s_n$ from right side of $s_{n-1}$ to left side of $s_{n-1}$, we have
\begin{align*}
 \prod_{i=1}^ns_i=s_1\cdots s_{n-2}(d_{2,1}s_n)s_{n-1}
\end{align*}
where $d_{2,1}=[s_{n-1}^{-1},s_n^{-1}]\in S_2'$ if not trvial.

Next, we move $d_{2,1}s_n$ from right side of $s_{n-2}$ to left side of $s_{n-2}$. That is:
\begin{align*}
 s_1\cdots s_{n-2}(d_{2,1}s_n)s_{n-1}=s_1\cdots s_{n-3}(d_{2,1}'d_{2,1}d_{3,1}s_n)s_{n-2}s_{n-1}
\end{align*}
where $d_{3,1}=[s_{n-2}^{-1},s_n^{-1}]\in S_2'$ and $d_{2,1}'=[s_{n-2}^{-1},\,d_{2,1}^{-1}]\in S'_3$ if they are not trivial.

We continue this process. In process of step $i$, we have a form
\begin{align*}
 s_1s_2\cdots s_{n-i+1}(e_{1}\cdots e_{j(i)}s_n)s_{n-i+2}\cdots s_{n-1}
\end{align*}
where $e_j\in S'$. We need to move the product $e_{1}\cdots e_{j(i)}s_n$ from right side of $s_{n-i+1}$ to left side of $s_{n-i+1}$. That is, we get
\begin{align*}
 &s_1s_2\cdots s_{n-i+1}(e_{1}\cdots e_{j(i)}s_n)s_{n-i+2}\cdots s_{n-1}\\
 &=s_1s_2\cdots s_{n-i}(e'_{1}e_1\cdots e'_{j(i)}e_{j(i)}d_{i+1,1}s_n)s_{n-i+1}\cdots s_{n-1}
\end{align*}
where $d_{i+1,1}=[s_{n-i+1}^{-1},s_n^{-1}]\in S_1'$ and $e_j'=[s_{n-2}^{-1},\,e_j^{-1}]$. Note that the length of $e_j'$ is equal to $1$ plus that of $e_j$ if not trivial.

We denote the number of elements of length $k$ in the form $e'_{1}e_1\cdots e'_{j(i)}e_{j(i)}d_{i+1,1}$ by $\beta_{k,i}$. For examples, for the word $e_1e_2e_1e_3e_4$ where $e_1,e_3\in S_2'$ and $e_1,e_4\in S_3'$, $\beta_{2,i}=3$, $\beta_{3,i}=2$. For repeating elements, we count the number as if they are different elements.
Then we have
\begin{align*}
 \beta_{2,i}\leq\beta_{2,i-1}+1,\quad \beta_{k,i}\leq\beta_{k,i-1}+\beta_{k-1,i-1}, \quad \forall k\leq \mathfrak{n}.
\end{align*}
Above relations show that we get a polynomil $p$ such that
\begin{align*}
 \sum_{k=2}^\mathfrak{n}\beta_{k,i}\leq p(n),\qquad \forall i\in\NN.
\end{align*}
Then we finish the proof.
\end{proof}

\begin{lemma}\label{le:11}
Let $\mathcal{H}$ be a nilpotent subgroup in $GL(N,\ZZ)$. Suppose for any element $g\in \mathcal{H}$, all eigenvalues of $g$ are of absolute value $1$. Let $S$ be a finite set in $\mathcal{H}$. Then there exists a polynomial $p$, such that for any product $\prod_{i=1}^ns_i$ where $s_i\in S$, $\norm{\prod_{i=1}^ns_i}\leq p(n)$.
\end{lemma}
\begin{proof}
We prove by using induction. Denote by $S=\{s_1,\cdots,s_d\}$ and set $S'=\{[s_i,\cdots[s_1,s_2]\cdots]:s_i\in S\}$.  If $S\subset\mathcal{H}_1$, since $\mathcal{H}_1$ is abelian we can write
\begin{align*}
 \prod_{i=1}^n s_{j_i}=s_1^{k_1}\cdots s_{d}^{k_d}
\end{align*}
where $\sum_{i=1}^j k_i=n$ and $s_1^{k_1}\cdots s_{d}^{k_j}$ is a reordered product of $\prod_{i=1}^n s_{j_i}$. Since each $s_i$ only has polynomial growth rate, there exist $C_S>0$ such that
\begin{align*}
 \norm{\prod_{i=1}^ns_i}\leq C_S\Pi_{i=1}^s(\abs{k_i}+1)^{N}\leq C_S (n+1)^{sN}.
\end{align*}
Then we proved the case of $S\subset\mathcal{H}_1$.

Suppose the conclusion holds for any $S\subset\mathcal{H}_i$. Next, we will prove the case when $S\subset\mathcal{H}_{i+1}$. By Lemma \ref{le:6} we can write
\begin{align*}
 \prod_{i=1}^ns_{j_i}=rs_1^{k_1}\cdots s_{d}^{k_d}
\end{align*}
where $\sum_{i=1}^j k_i=n$, $s_1^{k_1}\cdots s_{d}^{k_d}$ is a reordered product of $\prod_{i=1}^n s_{j_i}$ and $r$ is a product of elements in $S'$ with word length bounded by $f(n)$ for a polynomial $f$ determined by $S$. Since $S'\subset \mathcal{H}_{i}$, by assumption there exists a polynomial $f_1$ determined by $S'$ such that
\begin{align*}
  \norm{r}\leq C_S f_1(f(n)).
\end{align*}
Hence we get
\begin{align*}
 \norm{\prod_{i=1}^ns_{j_i}}\leq \norm{s_1^{k_1}\cdots s_{d}^{k_d}} \cdot\norm{r}\leq C_S (n+1)^{sN}f_1(f(n)),
\end{align*}
which implies the conclusion for the case of $S\subset\mathcal{H}_{i+1}$. Then we finish the proof.
\end{proof}

By using the two lemmas, we can proceed to the proof of the following:
\begin{proposition}\label{po:1}
Suppose $\mathcal{H}$ has a lower central series of length $\mathfrak{n}$. Then
\begin{enumerate}
  \item \label{for:43}all Lyapunov exponents of $\alpha(z)$ are $0$ if $z\in \mathcal{H}_{\mathfrak{n}-1}$;

  \smallskip

  \item \label{for:49}for any $x,\,y\in \mathcal{H}$, $\alpha(x)$ and $\alpha(y)$ preserve Lyapunov spaces of each other;

  \smallskip

  \item \label{for:50}the Lyapunov exponents of $\alpha(xy)$ is sum of corresponding Lyapunov exponents of $x$ and $y$.

  \smallskip
  \item \label{for:115}let $S=\{s_1,\cdots, s_d\}$ be a finite set in $\mathcal{H}$. Then there exists a polynomial $p$ such that
  \begin{align*}
   C_S'p(n)^{-1}&\leq \norm{s_1^{k_1}\cdots s_d^{k_d}(\prod_{i=1}^n s_{j_i})^{-1}}_{\text{min}}\\
   &\leq\norm{s_1^{k_1}\cdots s_d^{k_d}(\prod_{i=1}^n s_{j_i})^{-1}}\leq C_Sp(n)
  \end{align*}
  where $s_1^{k_1}\cdots s_d^{k_d}$ is a reordered product of $\prod_{i=1}^n s_{j_i}$.
\end{enumerate}
\end{proposition}
\begin{proof}
In this part  we identify $\alpha(x)$ and $x$ for any $x\in\mathcal{H}$ if there is no confusion. We just need to prove the first three statements. \eqref{for:115} follows from \eqref{for:43}, Lemma \ref{le:6} and \ref{le:11} immediately.

Obviously, \eqref{for:43} holds for any $z\in\mathcal{H}_0$; and \eqref{for:49} and \eqref{for:50} hold if $D_1(x,y)\in \mathcal{H}_0$ (see \eqref{for:103} of Section \ref{sec:1}).

Suppose \eqref{for:43} holds for any $z\in\mathcal{H}_i$; and \eqref{for:49} and \eqref{for:50} hold for any $x,\,y$ if $D_1(x,y)\in \mathcal{H}_{i}$, $i< \mathfrak{n}-1.$ Next, firstly we will show that \eqref{for:43} holds for any $z\in\mathcal{H}_{i+1}$. Suppose $z=D_1(z_1,z_2)$ for some $z_1,\,z_2\in\mathcal{H}$.
Inductively we can show that for any $n\in\ZZ$
\begin{align*}
 z_1z_2^n=z_2^nz_1D_1(z_1,z_2)^nf_n,\qquad f_n\in \mathcal{H}_{i}.
\end{align*}
Then by assumption there exists a a full measure set $\Gamma_{z_1,z_2,z}$ such that the Lyapunov exponents of $z_1D_1(z_1,z_2)^nf_n$ are of the form $\lambda+n\mu$ where $\lambda$ and $\mu$ are corresponding Lyapunov exponents of $z_1$ and $D_1(z_1,z_2)$ since all Lyapunov exponents of $f_n$ are $0$ by assumption. The fact that  $z_1D_1(z_1,z_2)^nf_n$ are conjugated with $z_1$ for all $n$ means there exists a a full measure set $\Gamma_{z_1,z_2,z}'$ such that all Lyapunov exponents of $D_1(z_1,z_2)=z$ are $0$. Since $\mathcal{H}_{i+1}$ is generated by such $z$ who is a commutator of a pair of elements in $\mathcal{H}$, then by assumption it follows that \eqref{for:43} holds for any $z\in\mathcal{H}_{i+1}$.

Finally, we will show that \eqref{for:49} and \eqref{for:50} hold for any $x,\,y$ if $D_1(x,y)\in \mathcal{H}_{i+1}$.
Note that $D_1(x,D_1(x,y))\in\mathcal{H}_i$, by assumption the Lyapunov exponents of $xD_1(x,y)$ are the sum of corresponding Lyapunov exponents of $x$ and $D_1(x,y)$. As we just showed that all Lyapunov exponents of $D_1(x,y)$ are $0$, then we see that $xD_1(x,y)$ and $x$ have the exactly the same Lyapunov spaces. Hence the relation $xy=yxD_1(x,y)$ implies that $y$ preserves each Lyapunov space of $x$. Also, relation $yx=xyD_1(x,y)^{-1}$ implies that $x$ preserves each Lyapunov space of $y$. Then we proved \eqref{for:49} in the case of $D_1(x,y)\in \mathcal{H}_{i+1}$.

Since $D_1(x,y)\in \mathcal{H}_{i+1}$, $S'=\{[x,\cdots[y,x]\cdots]: \text{ for all lenth}\}$ is in $\mathcal{H}_{i+1}$. By Lemma \ref{le:6} we can write
\begin{align}\label{for:51}
(xy)^{k}=e_kx^{k}y^{k}
\end{align}
where $e_k$ is a product of elements in $S'$ with word length bounded by $p(k)$, where $p$ is a polynomial determined by $x,\,y$. Since all elements in $\mathcal{H}_{i+1}$ are with all Lyapunov exponents $0$ as we proved, by  Lemma \ref{le:11} we get
\begin{align}\label{for:114}
 \norm{e_k}\leq p_1(\abs{k}),\qquad \forall k\in\ZZ.
\end{align}
For $y$ we have a decomposition:
\begin{align}\label{for:38}
  \RR^N=\bigoplus_{i\in I} J_{\mu_i}
\end{align}
where $J_{\mu_i}$ is the Lyapunov space of $y$ with Lyapunov exponent $\mu_i$.

Since $x$ preserves Lyapunov spaces $y$, each $J_{\mu_i}$ is $x$-invariant. Then we have a decomposition for each $J_{\mu_i}$:
\begin{align*}
  J_{\mu_i}=\bigoplus_{j\in J_i} J_{\lambda_{j(i)},\mu_{i}}
\end{align*}
such that each $J_{\lambda_{j(i)},\mu_{i}}$ is a Lyapunov space of $x$ on $J_{\mu_i}$ with Lyapunov exponent $\lambda_{j(i)}$.

Using \eqref{for:51}, \eqref{for:114} and \eqref{for:39} for any $v\in J_{\lambda_{j(i)},\mu_{i}}$ we have
\begin{align*}
 \norm{(xy)^{k}v}&=\norm{e_kx^{k}y^{k}v}\leq C_{x,y}p_1(k)e^{(\lambda_{j(i)}+\mu_i)k}(\abs{k}+1)^{N}\norm{v}
 \end{align*}
for any $k>0$.

It follows that
\begin{align*}
  \lim_{k\rightarrow+\infty }k^{-1}\log \norm{(xy)^{k}\mid_{J_{\lambda_{j(i)},\mu_{i}}}}\leq \lambda_{j(i)}+\mu_i
\end{align*}
On the other hand, applying similar reasoning we can show
\begin{align*}
  \lim_{k\rightarrow+\infty }k^{-1}\log \norm{(xy)^{-k}\mid_{J_{\lambda_{j(i)},\mu_{i}}}}\leq -\lambda_{j(i)}-\mu_i,
\end{align*}
which implies:
\begin{align*}
  &\lim_{k\rightarrow+\infty }k^{-1}\log \norm{(xy)^{k}\mid_{J_{\lambda_{j(i)},\mu_{i}}}}_{\text{min}}\\
  &=\lim_{k\rightarrow+\infty }(-k)^{-1}\log \norm{(xy)^{-k}\mid_{J_{\lambda_{j(i)},\mu_{i}}}}\geq\lambda_{j(i)}+\mu_i.
\end{align*}
This shows that the Lyapunov exponent of $xy$ on $J_{\lambda_{j(i)},\mu_{i}}$ is $\lambda_{j(i)}+\mu_i$. Then we proved \eqref{for:50} in case of $D_1(x,y)\in \mathcal{H}_{i+1}$.

\end{proof}
\begin{remark}\label{re:3}
It is a result of Kronecker \cite{Kronecker} which
states that an integer matrix with all eigenvalues on the unit circle has to have
all eigenvalues roots of unity. Then \eqref{for:43} implies all elements in $\alpha([\mathcal{H},\mathcal{H}])$ are not ergodic.
\end{remark}
Then next two corollaries are simple, but will be frequently used in the subsequent part of this section.
\begin{corollary}\label{cor:1}
\begin{enumerate}
  \item \label{for:54}For any $y\in \mathcal{H}$ and any Lyapunov space $V$ of $A$ (see \eqref{for:44} of Section \ref{sec:1}) we have
\begin{align*}
 \norm{A^{n}\bar{y}A^{-n}\mid_V}&\leq C(\abs{n}+1)^{2N}\norm{\bar{y}\mid_V},\quad \forall n\in\ZZ\backslash0.
 \end{align*}
 \item \label{for:106} for any $n\in\ZZ\backslash0$ and $1\leq i\leq \mathfrak{n}-1$,
 \begin{align*}
\norm{D_i(A^n,\bar{y})^{\pm 1}}\leq C\norm{\bar{y}^{\pm 1}}^{2^i}\abs{n}^{2^{i+1}N}.
 \end{align*}
 \end{enumerate}

\end{corollary}
\begin{proof}
On any Lyapunov space $V$, choose a basis in which $A$ has its Jordan normal form $xJ$, where $x$ is diagonal and $J$ is unipotent. It is clear that
\begin{align*}
  \norm{x^{n}\bar{y}x^{-n}\mid_V}\leq C_A\norm{\bar{y}\mid_V},\qquad \forall n\in\ZZ.
\end{align*}
By using \eqref{for:39} for any $n\in\ZZ\backslash0$ we have
\begin{align*}
 \norm{A^{n}\bar{y}A^{-n}\mid_V}&\leq C_A\norm{\bar{y}\mid_V}\cdot\norm{J^n}^2\leq C_A\abs{n}^{2N}\norm{\bar{y}\mid_V}.
 \end{align*}
 \eqref{for:106} is a direct consequence of \eqref{for:54}.
\end{proof}
For any $y\in \mathcal{H}$ and $x\in [\mathcal{H},\mathcal{H}]$, if $\bar{y}$ is Anosov then $\overline{xy}$ is Anosov, which is a direct consequence of Proposition \ref{po:1}. If $\bar{y}$ is ergodic, the next result shows that $\overline{xy}$ is also ergodic.
\begin{corollary}\label{cor:5}

For any $y\in \mathcal{H}$ and $x\in [\mathcal{H},\mathcal{H}]$, if $\bar{y}$ is ergodic then $\overline{xy}$ is also ergodic.
\end{corollary}
\begin{proof}
If $xy$ is not ergodic, there exists a vector $v\neq 0$ such that $(xy)^{m}v=v$, $m>0$. For any $0\neq u\in J_{0,\mu_{i}}$ (all Lyapunov exponents of $x$ are $0$ by Proposition of \ref{po:1}), using \eqref{for:51} and \eqref{for:115} of Proposition of \ref{po:1} we have
\begin{align*}
 \norm{(xy)^{k}u}&=\norm{e_kx^{k}y^{k}u}\geq C_{x,y}e^{\mu_ik}(k+1)^{-N}p(k)^{-1}\norm{u}
\end{align*}
for any $k\in\NN$ and a polynomial $p$.  It follows that
\begin{align*}
  \lim_{k\rightarrow\infty}k^{-1}\log\norm{(xy)^{k}u}\geq\mu_i.
\end{align*}
Since $\pi_1(v)\neq 0$ by ergodicity of $\bar{y}$ we have
\begin{align*}
  \lim_{k\rightarrow\infty}k^{-1}\log\norm{(xy)^{km}\pi_1(v)}>0.
\end{align*}
On the other hand, since
\begin{align*}
 k^{-1}\log\norm{v} =k^{-1}\log\norm{(xy)^{km}v}\geq k^{-1}\log C\norm{(xy)^{km}\pi_1(v)},
\end{align*}
where $C$ is a constant only dependent on $x$ and $y$, we get a contradiction
\begin{align*}
  0=\lim_{k\rightarrow\infty}k^{-1}\log\norm{v}=\lim_{k\rightarrow\infty}k^{-1}\log\norm{(xy)^{km}\pi_1(v)}>0.
\end{align*}
Hence we get the conclusion.
\end{proof}

At the end of this section, we make a slight digression to prove following results, whose role will be clear in Section \ref{sec:2}. Above corollary shows that as $n$ increases, the norm of $D_i(A^n,\bar{y})$ increases polynomially, while that of $A^n$ increase or deceases exponentially along hyperbolic directions. Then by increasing $n$, we can let the set $\{\mathcal{M}_{A^n}(D_i(A^n,\bar{y})v):v\in E_{A^n}\}$ (see \eqref{for:105} of Section \ref{sec:1}) be at most one step (future or past) away from $E_{A^n}$.
\begin{lemma}\label{le:9}
For any $c>0$ there exists $N_1(c)\in\NN$ such that for any $n\geq N_1$ and any $x\in\mathcal{H}$ with $\abs{\bar{x}}\leq c$, if $v=\mathcal{M}_{A^n}(v)$ then
\begin{align*}
  \mathcal{M}_{A^n}\big(D_i(A^n,\bar{x})v\big)=A^{nj}D_i(A^n,\bar{x})v,\qquad j=0, \pm 1
\end{align*}
for any $1\leq i\leq \mathfrak{n}$.
\end{lemma}
\begin{proof}
%Using \eqref{for:49} of Proposition \ref{po:1} we get decompositions $A=x_1x_2$ and $B=py_1y_2p^{-1}$ satisfying conditions \eqref{for:52}, \eqref{for:53} and \eqref{for:48} in the proof of Proposition \ref{po:1}. Then we have
%\begin{align*}
%  &A^{2n}D_1(A^n,B)A^{-n}=A^{n}B^{-1}A^nBA^{-n}\\
%  &=x_1^{n}x_2^{n}py_1^{-1}y_2^{-1}p^{-1}(x_1^{n}x_2^{n}py_1y_2p^{-1}x_1^{-n})x_2^{-n}.
%\end{align*}
%Applying the same reasoning for $i=1,2,3$ we obtain
%\begin{align}\label{for:55}
% &\abs{A^{2n}D_1(A^n,B)A^{-n}}_{V_i(A)}\notag\\
% &\leq \abs{x_1^{n}x_2^{n}py_1^{-1}y_2^{-1}p^{-1}}_{V_i(A)}\cdot\abs{x_1^{n}x_2^{n}py_1y_2p^{-1}x_1^{-n}}_{V_i(A)}\cdot\abs{x_2^{-n}}_{V_i(A)}\notag\\
% &\leq C\abs{n}^{3N}\abs{x_1^{n}}_{V_i(A)}\notag\\
% &\leq C\abs{n}^{4N}\abs{A^{n}}_{V_i(A)}.
%\end{align}
%In the last inequality we use assumption \eqref{for:53}.

%For $i=1,2,3$ argue in a similar way we  get
%\begin{align}
%\abs{A^{-2n}D_1(A^n,B)A^{n}}_{V_i(A)}&\leq C\abs{n}^{4N}\abs{A^{-n}}_{V_i(A)}\label{for:56};
%\end{align}
%we also get the lower bounds for minimal norm:
%\begin{align}
% \abs{{A^{-n}D_1(A^n,B)^{-1}A^{2n}}}^{-1}_{V_i(A)}&\geq C\abs{n}^{-4N}\abs{A^{n}}^{-1}_{V_i(A)}\label{for:57},\\
% \abs{A^{n}D_1(A^n,B)^{-1}A^{-2n}}^{-1}_{V_i(A)}&\geq C\abs{n}^{-4N}\abs{A^{-n}}^{-1}_{V_i(A)}\label{for:54}.
%\end{align}
To simply notion we use $d_n(i)$ to denote $D_i(A^n,\bar{x})$. If
\begin{align*}
 \norm{\pi_3(d_n(i)v)}\geq \norm{\pi_{1,2}(d_n(i)v)},
\end{align*}
(see \eqref{for:111} of Section \ref{sec:1}), for any $n>0$ we have
\begin{align}\label{for:58}
&\frac{\abs{\pi_{1,2}(A^{2n}d_n(i)v)}}{\abs{\pi_3(A^{2n}d_n(i)v)}}\stackrel{(1)}{=} \frac{\abs{A^{2n}d_n(i)A^{-n}\pi_{1,2}(A^nv)}}{\abs{A^{2n}d_n(i)A^{-n}\pi_3(A^nv)}}\notag\\
&\geq\frac{\abs{A^{2n}d_n(i)A^{-n}\mid_{V_{1,2}A}}_{\text{min}}\cdot\abs{\pi_{1,2}(A^nv)}}{\norm{A^{2n}d_n(i)A^{-n}\mid_{V_{3}A}}\cdot
\abs{\pi_3(A^nv)}}\notag\\
&\geq\frac{\abs{A^{n}\mid_{V_{1,2}A}}_{\text{min}}\cdot\abs{A^{n}d_n(i)A^{-n}\mid_{V_{1,2}A}}_{\text{min}}
\cdot\abs{\pi_{1,2}(A^nv)}}{\abs{A^{n}\mid_{V_{3}A}}\cdot\abs{A^{n}d_n(i)A^{-n}\mid_{V_{3}A}}\cdot
\abs{\pi_3(A^nv)}}\notag\\
&\stackrel{(2)}{\geq}\frac{C\abs{n}^{-2N}\abs{A^{n}\mid_{V_{1,2}(A)}}_{\text{min}}
\cdot\norm{d_n(i)}_{\text{min}}\cdot\abs{\pi_{1,2}(A^nv)}}{\abs{n}^{2N}\abs{A^{n}\mid_{V_3(A)}}\cdot\norm{d_n(i)}\cdot
\abs{\pi_3(A^nv)}}\notag\\
&\stackrel{(3)}{\geq}\frac{C_c\abs{n}^{-(3+2^{\mathfrak{n}+1})N}\cdot\abs{\pi_{1,2}(A^nv)}}{\abs{n}^{(3+2^{\mathfrak{n}+1})N}
\rho^{-n}\abs{\pi_3(A^nv)}}\notag\\
&\stackrel{(4)}{>} C_c\abs{n}^{-(6+2^{\mathfrak{n}+2})N}\rho^n>1
\end{align}
proving $n$ is big enough. Here $\rho$ is defined in \eqref{for:4} of Section \ref{sec:1}

Here $(1)$ is from the fact that $d_n(i)$ preserves Lyapunov spaces of $A$; in $(2)$ since
\begin{align*}
  \abs{A^{n}d_n(i)A^{-n}\mid_{V_{1,2}A}}_{\text{min}}=\abs{A^{n}d_n(i)^{-1}A^{-n}\mid_{V_{1,2}A}}^{-1}
\end{align*}
\eqref{for:54} of Corollary \ref{cor:1} shows:
\begin{align*}
  \abs{A^{n}d_n(i)A^{-n}\mid_{V_{1,2}A}}_{\text{min}}&\geq C\abs{n}^{-2N}\norm{d_n(i)^{-1}}^{-1}=C\abs{n}^{-2N}\norm{d_n(i)}_{\text{min}}\\
  &\geq C\abs{\bar{x}^{-1}}^{2^\mathfrak{n}}\abs{n}^{-(2^{i+1}+1)N}\geq C_c\abs{n}^{-(2^{i+1}+1)N}.
\end{align*}
Hence we get $(2)$; in $(3)$ we use the observation:
\begin{align*}
  \abs{A^{n}\mid_{V_{1,2}A}}_{\text{min}}\geq \abs{A^{n}\mid_{V_{2}A}}_{\text{min}}\geq C\norm{n}^{-N};
\end{align*}
in $(4)$ we use the fact $\frac{\abs{\pi_{1,2}(A^nv)}}{\abs{\pi_3(A^nv)}}>1$.

This implies that in this condition we get
\begin{align*}
  \mathcal{M}_{A^n}\big(D_i(A^n,\bar{x})v\big)=A^{nj}D_i(A^n,\bar{x})v,\qquad j=0, 1.
\end{align*}
If $\abs{\pi_3(d_n(i)v)}<\abs{\pi_{1,2}(d_n(i)v)}$, then
\begin{align}\label{for:59}
&\frac{\abs{\pi_{1,2}(A^{-n}d_n(i)v)}}{\abs{\pi_3(A^{-n}d_n(i)v)}}= \frac{\abs{A^{-n}d_n(i)\pi_{1,2}(v)}}{\abs{A^{-n}d_n(i)\pi_3(v)}}\notag\\
&\leq\frac{\norm{A^{-n}d_n(i)\mid_{V_{1,2}A}}\cdot\abs{\pi_{1,2}(v)}}{\norm{A^{-n}d_n(i)\mid_{V_{3}A}}_{\text{min}}
\cdot\abs{\pi_3(v)}}\notag\\
&\leq\frac{C_c\abs{n}^{2N}\abs{n}^{2^{\mathfrak{n}+1}N}\cdot\abs{\pi_{1,2}(v)}}{\abs{n}^{-2^{\mathfrak{n}+1}N}\rho^n\abs{ \pi_3(v)}}\notag\\
&\leq C_c\abs{n}^{(2^{\mathfrak{n}+2}+2)N}\rho^{-n}<1
\end{align}
providing $n$ is big enough.

In the last step we used $\frac{\abs{\pi_{1,2}(v)}}{\abs{\pi_3(v)}}\leq1$. Then in this condition we have
\begin{align*}
  \mathcal{M}_{A^n}\big(D_i(A^n,\bar{x})v\big)=A^{nj}D_i(A^n,\bar{x})v,\qquad j=-1.
\end{align*}
similarly, we obtain

Let  $N_1$ be the integer which satisfies the inequalities
\begin{align*}
  C_c\abs{N_1}^{(2^{\mathfrak{n}+2}+2)N}\rho^{-N_1}<1\quad\text{ and }\quad C_c\abs{N_1}^{-(6+2^{\mathfrak{n}+2})N}\rho^{N_1}>1.
\end{align*}
Then \eqref{for:58} and \eqref{for:59} show that $N_1$ what we need.

\end{proof}
For any $v\in\ZZ^N$ if $\frac{\pi_{1,2}(v)}{\pi_3(v)}>1$, then $v\notin E_{A^n}$ for any $n\in\ZZ\backslash0$; on the other hand,  conditions $\frac{\pi_{1,2}(v)}{\pi_3(v)}<1$ and $\frac{\pi_{1,2}(A^{n}v)}{\pi_3(A^{n}v)}<1$ should be satisfied at the same time to guarantee $v\notin E_{A^n}$ for any $n\in\NN$.
The next lemma lists several criterions to tell wether a vector is in $E_{A^n}$ or nor for big enough $n$.

\begin{lemma}\label{cor:3}For any $c>0$ there exists $N_2(c)\in\NN$ such that for any $n\geq N_2$, any $x\in\mathcal{H}$ with $\abs{\alpha(x)}\leq c$, if denote
$D_i(A^n,\bar{x})$ by $d_{n}(i)$ and set
 \begin{align*}
   A(n,l_1,l_2,l_3)_{i_1,i_2}^{j_1,j_2,j_3}= d_{n}(i_1)(A^nd_{n}(l_1))^{j_1}d_{n}(i_2)(A^nd_{n}(l_2))^{j_2}(A^nd_{n}(l_3))^{j_3},
 \end{align*}
where $1\leq i_1,i_2,l_1,l_2,l_3\leq \mathfrak{n}$ and $j_1,j_2,j_3\in\ZZ$, then:
\begin{enumerate}
  \item \label{for:42}for $v\in E_{A^n}$
  \begin{enumerate}
    \item [a)] if $j_1,j_2,j_3\geq 0$ (resp. $j_1,j_2,j_3\leq0$) satisfying $\sum_{i=1}^3\abs{j_i}\geq 2$, then
    \begin{align*}
      A(n,l_1,l_2,l_3)_{i_1,i_2}^{j_1,j_2,j_3}v\notin E_{A^n};
    \end{align*}

    \item [b)] if $j_1,j_2,j_3\geq 0$ (resp. $j_1,j_2,j_3\leq0$) and $\sum_{i=1}^3\abs{j_i}\geq 3$, and
    if
    \begin{align*}
     A^{nz}A(n,l_1,l_2,l_3)_{i_1,i_2}^{j_1,j_2,j_3}v\in E_{A^n}
    \end{align*}
    for some $z\in\ZZ$,  then $z\leq -2$ (resp. $z\geq2$).
  \end{enumerate}

      \smallskip
  \item \label{for:68}if $A^{nm}v\in E_{A^n}$ $m\geq 1$ (resp. $m\leq-2$), then:
  \begin{enumerate}
    \item [a)] $d_{n,i_1}(A^nd_{n,l_1})^{j_1}v\notin E_{A^n}$ if $j_1\leq-1$ (resp. $j_1\geq0$);

    \smallskip

    \item [b)] if $A^{nz}d_{n,i_1}(A^nd_{n,l_1})^{j_1}v\in E_{A^n}$ for some $z\in\ZZ$ where $j_1\leq-1$ (resp. $j_1\geq1$) then $z\geq1$ (resp. $z\leq-2$).
  \end{enumerate}

\end{enumerate}

\end{lemma}
\begin{proof} For $A$ we have the following Lyapunov space decomposition:
\begin{align*}
  \RR^N=\bigoplus_{i\in I} J_{\mu_i}
\end{align*}
where $J_{\mu_i}$ is the Lyapunov space of $A$ with Lyapunov exponent $\mu_i$. We choose a basis of $\RR^N$ in which $A=x_1x_2$, where $x_1$ is diagonal and $x_2$ is unipotent, $x_1$ and $x_2$ commute, and the eigenvalues of $x_1$ coincides with that of $A$. For any matrix $y$ preserving Lyapunov space spaces of $A$ we have
\begin{align}\label{for:40}
 \norm{x_1^nyx_1^{-n}} \leq C_A\norm{y},\qquad \forall n\in\ZZ.
\end{align}
Noting $x_2$ is unipotent and using Corollary \ref{cor:1} we have
\begin{align}\label{for:41}
 C_c(\abs{n}+1)^{-(1+2^{\mathfrak{n}+1})\abs{j}N}\leq\norm{(x_2^nd_n(i))^j}\leq C_c(\abs{n}+1)^{(1+2^{\mathfrak{n}+1})\abs{j}N}.
\end{align}
Then \eqref{for:40} and \eqref{for:41} imply:
\begin{align*}
  &\norm{A^kA(n,l_1,l_2,l_3)_{i_1,i_2}^{j_1,j_2,j_3}A^m\mid_{J_{\mu_i}}}\\
  &\leq C_c^{\abs{\mathfrak{j}}}e^{(n\mathfrak{j}+m+k)\mu_i}(\abs{m}+1)^N(\abs{k}+1)^N(\abs{n}+1)^{2+(1+2^{\mathfrak{n}+1})\abs{\mathfrak{j}}N}
\end{align*}
and
\begin{align*}
  &\norm{A(n,l_1,l_2,l_3)_{i_1,i_2}^{j_1,j_2,j_3}A^m\mid_{J_{\mu_i}}}\\
  &\geq C_c^{\abs{\mathfrak{j}}}e^{(n\mathfrak{j}+m)\mu_i}(\abs{m}+1)^{-N}(\abs{k}+1)^{-N}(\abs{n}+1)^{-2-(1+2^{\mathfrak{n}+1})\abs{\mathfrak{j}}N}
\end{align*}
for any $J_{\mu_i}$.  Here $\mathfrak{j}=j_1+j_2+j_3$ and $\mathfrak{j}=\abs{j_1}+\abs{j_2}+\abs{j_3}$.

$(a)$ of \eqref{for:42}  for $j_1,j_2,j_3\geq0$: by applying above inequalities we have
\begin{align*}
 &\frac{\abs{\pi_{1,2}(A(n,l_1,l_2,l_3)_{i_1,i_2}^{j_1,j_2,j_3}v)}}{\abs{\pi_3(A(n,l_1,l_2,l_3)_{i_1,i_2}^{j_1,j_2,j_3}v)}}
 =\frac{\abs{A(n,l_1,l_2,l_3)_{i_1,i_2}^{j_1,j_2,j_3}A^{-n}\pi_{1,2}(Av)}}{\abs{A(n,l_1,l_2,l_3)_{i_1,i_2}^{j_1,j_2,j_3}A^{-n}\pi_3(Av)}}\\
 &\geq \frac{\abs{A(n,l_1,l_2,l_3)_{i_1,i_2}^{j_1,j_2,j_3}A^{-n}\mid_{V_{1,2}A}}_{\text{min}}}{\abs{A(n,l_1,l_2,l_3)_{i_1,i_2}^{j_1,j_2,j_3}A^{-n}\mid_{V_3 A}}}
 \frac{\abs{\pi_{1,2}(A^nv)}}{\abs{\pi_3(A^nv)}}\\
 &>C_c^{\mathfrak{j}}\rho^{n(\mathfrak{j}-1)}(\abs{n}+1)^{-6-(2+2^{2\mathfrak{n}+2})\mathfrak{j}N}>1
\end{align*}
providing $n$ is big enough. Here $\rho$ is defined in \eqref{for:4} of Section \ref{sec:1}. This implies the conclusion.

\smallskip
$(a)$ of \eqref{for:42}  for $j_1,j_2,j_3\leq0$: similar to the proof in $(a)$ we get
\begin{align*}
  &\frac{\abs{\pi_{1,2}(A(n,l_1,l_2,l_3)_{i_1,i_2}^{j_1,j_2,j_3}v)}}{\abs{\pi_3(A(n,l_1,l_2,l_3)_{i_1,i_2}^{j_1,j_2,j_3}v)}}\\
  &<\frac{\abs{A(n,l_1,l_2,l_3)_{i_1,i_2}^{j_1,j_2,j_3}\mid_{V_{1,2}A}}}{\abs{A(n,l_1,l_2,l_3)_{i_1,i_2}^{j_1,j_2,j_3}\mid_{V_3 A}}_{\text{min}}}
 \frac{\abs{\pi_{1,2}(v)}}{\abs{\pi_3(v)}}\\
 &<C_c^{-\mathfrak{j}}\rho^{n\mathfrak{j}}(\abs{n}+1)^{4+(2+2^{2\mathfrak{n}+2})\mathfrak{j}N}<1
\end{align*}
and
\begin{align*}
  &\frac{\abs{\pi_{1,2}(A^nA(n,l_1,l_2,l_3)_{i_1,i_2}^{j_1,j_2,j_3}v)}}{\abs{\pi_3(A^nA(n,l_1,l_2,l_3)_{i_1,i_2}^{j_1,j_2,j_3}v)}}\\
  &<\frac{\abs{A^nA(n,l_1,l_2,l_3)_{i_1,i_2}^{j_1,j_2,j_3}\mid_{V_{1,2}A}}}{\abs{A^nA(n,l_1,l_2,l_3)_{i_1,i_2}^{j_1,j_2,j_3}\mid_{V_3 A}}_{\text{min}}}
 \frac{\abs{\pi_{1,2}(v)}}{\abs{\pi_3(v)}}\\
 &<C_c^{-\mathfrak{j}}\rho^{n(\mathfrak{j}+1)}(\abs{n}+1)^{6+(2+2^{2\mathfrak{n}+2})\mathfrak{j}N}<1
\end{align*}
providing $n$ is big enough. Then we get the conclusion.

\smallskip
$(b)$ of \eqref{for:42}  for $j_1,j_2,j_3\geq0$: follow the proof line $(a)$ we get
\begin{align*}
 &\frac{\abs{\pi_{1,2}(A^{nz}A(n,l_1,l_2,l_3)_{i_1,i_2}^{j_1,j_2,j_3}v)}}{\abs{\pi_3(A^{nz}A(n,l_1,l_2,l_3)_{i_1,i_2}^{j_1,j_2,j_3}v)}}\\
 &>C_c^{\mathfrak{j}}\rho^{n(\mathfrak{j}-1+z)}(\abs{n}+1)^{-4-\abs{z}-(2+2^{2\mathfrak{n}+2})\mathfrak{j}N}>1
\end{align*}
providing $n$ is big enough and $z\geq-1$. Then we get the conclusion.

\smallskip
$(b)$ of \eqref{for:42}  for $j_1,j_2,j_3\leq0$: similarly, we get
\begin{align*}
  &\frac{\abs{\pi_{1,2}(A^{nz}A(n,l_1,l_2,l_3)_{i_1,i_2}^{j_1,j_2,j_3}v)}}{\abs{\pi_3(A^{nz}A(n,l_1,l_2,l_3)_{i_1,i_2}^{j_1,j_2,j_3}v)}}\\
  &<C_c^{-\mathfrak{j}}\rho^{n(\mathfrak{j}+z)}(\abs{n}+1)^{4+\abs{z}+(2+2^{2\mathfrak{n}+2})\mathfrak{j}N}<1
\end{align*}
and
\begin{align*}
  &\frac{\abs{\pi_{1,2}(A^nA^{nz}A(n,l_1,l_2,l_3)_{i_1,i_2}^{j_1,j_2,j_3}v)}}{\abs{\pi_3(A^nA^{nz}A(n,l_1,l_2,l_3)_{i_1,i_2}^{j_1,j_2,j_3}v)}}\\
  &<C_c^{-\mathfrak{j}}\rho^{n(\mathfrak{j}+1+z)}(\abs{n}+1)^{4+\abs{z}+(2+2^{2\mathfrak{n}+2})\mathfrak{j}N}<1
\end{align*}
providing $n$ is big enough and $z\leq 1$. Then we get the conclusion.

\eqref{for:68} is a direct consequence of \eqref{for:42} and its proof.
\end{proof}

\section{Orbit growth for the dual action} In this part we list several results about certain estimates of
the $C^r$ or Sobolev norms of specifically defined functions or maps if the exponential growth along individual orbits of the dual action are obtained.
\begin{lemma}[Lemma 4.3, \cite{Damjanovic4}]\label{le:7} Let $F_1,\,F_2$ be commuting integer matrices in $GL(N,\ZZ)$. Suppose there exist constant $C,\,\tau>0$ such that for every non-zero integer vector $v\in\ZZ^N$ and for any $k=(k_1,k_2)\in\ZZ^2$,
  \begin{align}\label{for:9}
\norm{F_1^{k_1}F_2^{k_2}v}\geq C\exp(\tau\abs{k})\norm{v}^{-N},
\end{align}
then:
\begin{enumerate}
 \item [a)]\label{for:2} For any $C^{\infty}$ function $\varphi$ on the torus $\TT^N$ and any $y=(y_1,y_2)\in\CC^2$
the following sums:
\begin{align*}
  S_K(\varphi,v,y,p)=\sum_{k=(k_1,k_2)\in K}y_1^{k_1}y_2^{k_2}\widehat{\varphi}_{F_1^{k_1}F_2^{k_2}v}
\end{align*}
converge absolutely for any $K\subset\ZZ^2$.
\smallskip

  \item [c)] Assume in addition to the assumptions in $b)$ that for a vector $n\in\ZZ^N$ and for
every $k\in K=K(v)\subset\ZZ^2$ we have
\begin{align}\label{for:96}
 p_1(\abs{k})\norm{F_1^{k_1}F_2^{k_2}v}\geq \norm{v}
\end{align}
where $p_1$ is a polynomial
  then we have
\begin{align*}
\abs{S_K(\varphi,v,y,p)}&\leq \sum_{k\in K}\abs{y_1}^{\abs{k_1}}\abs{y_2}^{\abs{k_2}}\abs{\widehat{\varphi}_{F_1^{k_1}F_2^{k_2}v}}\\
&\leq C_{a,\abs{y}^{\pm 1},\delta}\norm{\varphi}_a\norm{v}^{-a+\kappa_y}
\end{align*}
for any $a>\kappa_{y}\stackrel{\rm def}{=}\frac{N+1}{\tau}(\abs{\log\abs{y_1}}+\abs{\log\abs{y_2}})$.
\smallskip

  \item [d)] If the assumptions of $c)$ are satisfied for every $v\in\ZZ^N$, then the function
  \begin{align*}
   S(\varphi)\stackrel{\rm def}{=}\sum_{v\in\ZZ^N}S_{K(v)}(\varphi,v,y,p)e_v
  \end{align*}
is a $C^\infty$ function if $\varphi$ is. Moreover, the following norm comparison holds:
\begin{align*}
 \norm{S(\varphi)}_{C^r}\leq C_{r,\abs{y}^{\pm 1}}\norm{\varphi}_{r+\sigma}
\end{align*}
for any $r\geq 0$ and $\sigma>N+2+[\kappa_{y}]$.

\end{enumerate}
\end{lemma}
\begin{remark}\label{re:1}
If $F$ is ergodic, then for any $v\in\ZZ^N\backslash0$
\begin{align*}
  \norm{\pi_{1}v}\geq C_F\norm{v}^{-N}\quad\text{ and }\quad\norm{\pi_{3}v}\geq C_F\norm{v}^{-N}
\end{align*}
where $C_F$ is a constant only dependent on $F$ (see Lemma 4.1 of \cite{Damjanovic4} and \cite{Katznelson}).
$\tau$ can be chosen to be the growth rate in the hyperbolic direction corresponding to $F$.
\end{remark}
The next result follows immediately from above lemma:
\begin{corollary}\label{cor:2}
Suppose $P_i,\,F_i$, $i=1,\,2$ are integer matrices in $GL(N,\ZZ)$ and $P_1P_2=P_2P_1$, $F_1F_2=F_2F_1$. Denote the eigenvalues of $P$ by $y_{1},\cdots,y_{N}$. Let $y=\sum_{j=1}^2\abs{\log\norm{P_{j}}}$. If condition \eqref{for:9} is satisfied, then for any $C^{\infty}$ map $\varphi:\TT^N\rightarrow\RR^N$ we obtain
\begin{enumerate}
  \item \label{for:26}the following sums:
\begin{align*}
  S_K(\varphi,v)=\sum_{k=(k_1,k_2)\in K}P_1^{k_1}P_2^{k_2}\widehat{\varphi}_{F_1^{k_1}F_2^{k_2}v}
\end{align*}
converge absolutely for any $K\subset\ZZ^2$.

\smallskip
  \item \label{for:34} Assume in addition \eqref{for:96} holds, then
\begin{align*}
\abs{S_K(\varphi,v)}\leq C_{a,y^{\pm 1}}\norm{\varphi}_a\norm{v}^{-a+\kappa_{P,F}}
\end{align*}
for any $a>\kappa_{P,F}\stackrel{\rm def}{=}\frac{(N+1)}{\tau}y$.

\smallskip
\item \label{for:27}If the assumptions of $(2)$ are satisfied for every $v\in\ZZ^N$, then the function
  \begin{align*}
   S(\varphi)\stackrel{\rm def}{=}\sum_{v\in\ZZ^N}S_{K(v)}(\varphi,v)e_v
  \end{align*}
is a $C^\infty$ map if $\varphi$ is. Moreover, the following norm comparison holds:
\begin{align*}
 \norm{S(\varphi)}_{C^r}\leq C_{r,y^{\pm 1}}\norm{\varphi}_{r+\sigma}
\end{align*}
for any $r\geq 0$ and $\sigma>N+2+[\kappa_{P,F}]$.
\end{enumerate}
\end{corollary}
\begin{proof}
Since
\begin{align*}
  &\sum_{k=(k_1,k_2)\in K}\norm{P_1^{k_1}P_2^{k_2}\widehat{\varphi}_{F_1^{k_1}F_2^{k_2}v}}\\
  &\leq C_{P,F}\max_{\delta=1,2}\{\norm{P_1^\delta}^{\abs{k_1}}\}\max_{\delta=1,2}\{\norm{P_2^\delta}^{\abs{k_2}}\}\sum_{k=(k_1,k_2)\in K}\norm{\widehat{\varphi}_{F_1^{k_1}F_2^{k_2}v}},
\end{align*}
we get the conclusion immediately.
\end{proof}
In the subsequent
part we prove the exponential growth along individual orbits of ergodic elements. It may be viewed as a generalization of Lemma 4.3 in \cite{Damjanovic4} to higher rank non-abelian actions by toral
automorphisms. Recall $A$ and $B$ are defined in \eqref{for:44} of Section \ref{sec:1}.
\begin{lemma}\label{le:1}
There exist constant $C>0$ such that for every non-zero integer vector $v\in\ZZ^N$ and for any $k=(k_1,k_2)\in\ZZ^2\backslash 0$,
  \begin{align}\label{for:92}
\norm{A^{k_1}B^{k_2}v}\geq C\exp\{\tau(\abs{k_1}+\abs{k_2})\}\norm{v}^{-N}.
\end{align}
\end{lemma}
\begin{proof}
Let $S=\{A^{k_1}B^{k_2}:(k_1, k_2)\in\ZZ^2\}$. Proposition \ref{po:1} shows that the space $\RR^N$ is decomposed into a direct sum of $S$-invariant Lyapunov spaces:
\begin{align}\label{for:90}
 \RR^d=\bigoplus_{i\in I}\mathbb{M}_i.
\end{align}
where $I\subset\{1,\cdots,N\}$; and the Lyapunov exponents of $A^{k_1}B^{k_2}$ are
\begin{align*}
 \chi_i(k)=k_1\chi_{A,i}+k_2\chi_{B,i}, \qquad i\in I
\end{align*}
where $k=(k_1,k_2)$ and $\chi_{A,i}$ and $\chi_{B,i}$ are Lyapunov exponents of $A$ and $B$ on $\mathbb{M}_i$ respectively.

Let $f(t):=\max_i\chi_i(t)$, $t\in S^1$. Then $f(t)$ is continuous and achieves its minimum on $S^1$ at some point $t_0$. Next, we will show that $f(t_0)>0$.

If $f(t_0)\leq0$ then for all $i\in I$ we have $\chi_i(t)\leq0$.
Since $\sum_{i\in I}\chi_i(t)=0$ for all $t$, it follows that $\chi_i(t)=0$ for all $i\in I$
and consequently $f(t_0)=0$. This implies existence of a line $l$ in $\RR^2$ such that for all points on $l$
$\chi_i$, $i\in I$ take value zero. Then the line $l$ cannot contain any non-zero integer vectors $k=(k_1,k_2)\in\ZZ^2$ otherwise all Lyapunov exponents
of $A^{k_1}B^{k_2}$ are $0$, which contradicts the ergodiccity of  $A^{k_1}B^{k_2}$ (see Remark \ref{re:3}).  Then for any $\epsilon>0$ there exists $k=(k_1,k_2)\in\ZZ^2$ such that $A^{k_1}B^{k_2}$ has all its eigenvalues $\epsilon$-close to $1$. Since the trace must be integers, it is equal to $N$. This implies all eigenvalues of $A^{k_1}B^{k_2}$ are $1$, which also contradicts the ergodicity of $A^{k_1}B^{k_2}$. Therefore, $f(t_0)>0$.

Choose $(k_1,k_2)\in \ZZ^2$ such that the Lyapunov space decomposition of $s=A^{k_1}B^{k_2}$ coincide with \eqref{for:90}. Let the minimal polynomial of $s$ on $\RR^N$ be $p$. Then $p=\prod_jp_j^{m_j}$ where $p_j$ is irreducible over $\ZZ$. Furthermore, each $p_j$ is separable and any pairwise different $p_j$ and $p_i$ have no common eigenvalues since otherwise these irreducible polynomials would factor over $\QQ$, and since it is monic, by Gauss' lemma, it would factor over $\ZZ$, which is a contradiction.  Then $\RR^N$ is decomposed into a direct sum of $s$-invariant subspaces:
\begin{align*}
 \RR^N=\bigoplus_{j\in J}\mathbb{I}_j
\end{align*}
where $J\subset\{1,\cdots,N\}$ and the minimal polynomial of $s$ on each $\mathbb{I}_j$ is $p_j^{m_j}$. Then each $I_j$ is spanned by a subset of $\ZZ^N$. Note that

For each $\mathbb{I}_j$ we have a decomposition:
\begin{align*}
 \mathbb{I}_j=\bigoplus_{n\in J_j}\mathbb{I}'_n
\end{align*}
where $J_j\subset\{1,\cdots,N\}$ and $\mathbb{I}'_n$ are Lyapunov spaces of $s$ on $\mathbb{I}_j$. Note that $\sharp(J_j)\geq 2$ for each $j\in J$ by ergodicity of $s$.

For any $j_i$, $j\in J$ we note that $s\mid_{\mathbb{I}'_{j_i}}$ and $s\mid_{{\bigoplus_{n\in J_j-\{j_i\}}}\mathbb{I}'_{n}}$ have no common eigenvalues,
and also $(\bigoplus_{n\in J_j-\{j_i\}}\mathbb{I}'_{n})\bigcap \ZZ^N=\{0\}$ because of irreducibility of $p_j$. This shows $(\bigoplus_{j\in J}{\bigoplus_{n\in J_j-\{j_i\}}}\mathbb{I}'_{n})\bigcap \ZZ^N=\{0\}$.

For any $v\in\ZZ^N$ let $v(i)$ be a projection of $v$ to $\mathbb{M}_{i}$, $i\in I$. For each $\mathbb{M}_i$, $i\in I$ note that
\begin{align*}
\mathbb{M}_i=\bigoplus_{j_i\in J_j}\mathbb{I}_{j_i}\bigcap\mathbb{M}_i
\end{align*}
Then by Katznelson¡¯s lemma \cite[Lemma 3]{Katznelson}, there exists a constant $\gamma_{i}$ such that
\begin{align}
 \norm{v(i)}\geq d(v, \bigoplus_{j\in J}{\bigoplus_{n\in J_j-\{j_i\}}}\mathbb{I}'_{n})\geq \gamma_{i}\norm{v}^{-N},
\end{align}
where $d$ is the Euclidean distance and the constant $\gamma_{i}$ depends only on the Lyapunov spaces splitting \eqref{for:90} for $S$.

Using decomposition \eqref{for:90} there exists a basis under which $A$ and $B$ have decompositions:
\begin{align}\label{for:91}
  A=px_1x_2p^{-1}\quad\text{ and }\quad B=y_1y_2
\end{align}
where $x_i,\,y_i,p\in SL(N,\RR)$, $i=1,2$ which satisfy:
\begin{enumerate}
  \item \label{for:13} $y_1$ and $x_1$ are diagonal and $y_2$ and $x_2$ are unipotent;

  \smallskip

  \item\label{for:14} $x_i,\,y_i,p$ ($i=1,2$) preserve decomposition \eqref{for:90}; and  $y_1y_2=y_2y_1$, $x_1x_2=x_2x_1$. Furthermore,
  \begin{align}\label{for:15}
  C_{A,B}^{-1}\norm{g}\leq\norm{z_1^mgz_1^{-m}}\leq C_{A,B}\norm{g},\qquad \forall m\in\ZZ,
  \end{align}
where $z$ stands for $x_1$ and $y_1$ and $g$ is any matrix preserving decomposition \eqref{for:90}.
\end{enumerate}
Then for any $u\in \mathbb{M}_i$, $i\in I$, any $k=(k_1,k_2)\in \ZZ^2$ we have
\begin{align*}
\norm{A^{k_1}B^{k_2}u}&\geq C\exp\{\chi_i(k_1,k_2)\}(\abs{k_1}+1)^{-N}(\abs{k_2}+1)^{-N}\norm{u}.
\end{align*}
By previous argument there exists $i\in I$ such that
\begin{align*}
\chi_{i}(k_1,k_2)\geq Cf(t_0)(\abs{k_1}+\abs{k_2}).
\end{align*}
Then:
\begin{align*}
\norm{A^{k_1}B^{k_2}v}&\geq C\norm{A^{k_1}B^{k_2}v(i)}\geq C\exp\{\tau(\abs{k_1}+\abs{k_2})\}\norm{v(i)}^{-N}\\
&\geq C\exp\{\tau(\abs{k_1}+\abs{k_2})\}\norm{v}^{-N}
\end{align*}
where $\tau=\frac{Cf(t_0)}{2}$.  Hence we proved \eqref{for:92}.
\end{proof}

At the end of this part, we obtain crucial estimates for the polynomial growth along individual orbits of the dual action of non-ergodic elements. We also get tame estimates of $C^r$ or Sobolev norms of functions or maps similar to that defined in Lemma \ref{le:7} or Corollary \ref{cor:3}.
\begin{lemma}\label{le:10}
Let $F$ and $Q$ be integer matrices in $GL(n,\ZZ)$. Suppose $F$ is ergodic and $Q$ is unipotent such that $FQ=QF$.
Then:
\begin{enumerate}
  \item \label{for:17}there exists a constant $C(F,Q)>0$ such that for every integer vector $v\in\ZZ^n$ satisfying $Qv\neq v$ and for any $(k_1,k_2)\in\ZZ^2$,
  \begin{align*}
\norm{F^{k_1}Q^{k_2}v}\geq C\rho^{\abs{k_1}}\abs{k_2}^{\frac{1}{2}}\norm{v}^{-n_1}
\end{align*}
where $n_1=(2n+3)n$ and $\rho>1$ is the growth rate in the hyperbolic direction corresponding to $F$.

\smallskip

  \item \label{for:21}For any $C^{\infty}$ function $\varphi$ on the torus, any vector $v\in\ZZ^n$ satisfying $Qv\neq v$, any $y\in\RR$ and a polynomial $p_1$
the following sums:
\begin{align*}
  S_K(\varphi,v)=\sum_{k=(k_1,k_2)\in K}y^{k_1}p_1(\abs{k_2})\widehat{\varphi}_{F^{k_1}Q^{k_2}v}
\end{align*}
converge absolutely for any $K\subset\ZZ^2$.

\smallskip
  \item \label{for:22}Assume in addition that for any vector $v\in\ZZ^N$ and for
every $k=(k_1,k_2)\in K=K(v)\subset\ZZ^2$ we have $p(\abs{k_1})\norm{F^{k_1}Q^{k_2}v}\geq \norm{v}$ where $p$ is a polynomial, then
\begin{align*}
\abs{S_K(\varphi,v)}\leq C_{a,y}\norm{\varphi}_a\norm{v}^{-a+\kappa}
\end{align*}
for any $a>\kappa_{F,Q}=(n_1+1)(4+4\text{deg}(p_1)+\abs{\log_\rho \abs{y}})$.
\end{enumerate}

\end{lemma}
\begin{proof}
\textbf{Proof of \eqref{for:17}}: Since $Q$ is unipotent, there exists $j_0\in\NN$ such that $(F-I)^{j_0}=0$ while $(F-I)^{j_0-1}\neq 0$.
There also exists an integer matrix $P$ such that $J=P^{-1}QP$ has its Jordan normal form. Then for every vector $u\in\RR^n\backslash 0$, we can write
$u=\sum_{0\leq i\leq j}u_i$ where $j\leq j_0-1$, such that $u_j\neq 0$; and if $u_i\neq 0$ then $(Q-I)^{i+1}u_i=0$, while $(Q-I)^iu_i\neq 0$. Furthermore, if $u\in\ZZ^N$ then $\det(P)\cdot u_i\in\ZZ^n$. By using power form of a Jordan block, it is easy to see that
\begin{align}
 \norm{(Q^mu)_{j-1}}\geq C_Q(\abs{m}\norm{u_j}-\norm{u_{j-1}}),\qquad \forall m\in\ZZ.
\end{align}

Since $Qv\neq v$, we can write $v=\sum_{0\leq i\leq k}v_i$, $k\geq 1$. Let $v'=\det(P)\cdot v_{k}$. Above analysis shows that $v'\in \ZZ^n$. From Remark \ref{re:1} it follows that
\begin{align*}
 \min\{\norm{\pi_1(v')},\norm{\pi_3(v')}\}\geq \gamma\norm{v'}^{-n}
\end{align*}
for some $\gamma$ and for all $v'$. Therefore
\begin{align*}
\min\{\norm{\pi_1(v_k)},\norm{\pi_3(v_k)}\}\geq \gamma\abs{\det(P)}^{-1}\norm{v'}^{-n}\geq \gamma'\norm{v}^{-n}.
\end{align*}
where $\gamma'$ is a constant only dependent on $F$ and $Q$. Hence
\begin{align*}
\norm{F^{k_1}Q^{k_2}v}&\geq C_F\rho^{\abs{k_1}}\min\{\norm{\pi_1(Q^{k_2}v)},\norm{\pi_3(Q^{k_2}v)}\}\\
&\overset{(1)}{\geq}C_F\rho^{\abs{k_1}}\min\{\norm{(Q^{k_2}\pi_1(v))_{k-1}},\norm{(Q^{k_2}\pi_3(v))_{k-1}}\}\\
&\geq C_{F,Q}\rho^{\abs{k_1}}\min_{\delta=1,3}\{\abs{k_2}\cdot\norm{(\pi_\delta(v))_{k}}-\norm{(\pi_\delta(v))_{k-1}}\}\\
&\overset{(2)}{=} C_{F,Q}\rho^{\abs{k_1}}\min_{\delta=1,3}\{\abs{k_2}\cdot\norm{\pi_\delta(v_{k})}-\norm{\pi_\delta(v_{k-1})}\}\\
&\geq C_{F,Q}\rho^{\abs{k_1}}(\gamma'\norm{v}^{-n}\cdot\abs{k_2}-\norm{v}).
\end{align*}
Here  $(1)$ and $(2)$ follow from the fact that $F$ and $Q$ commute.

Then it follows that if $\abs{k_2}\geq 4\max\{1,\gamma'^{-2}\}\norm{v}^{2(n+1)}$, then
\begin{align}\label{for:20}
  \norm{F^{k_1}Q^{k_2}v}\geq C_{F,Q}\abs{k_2}^{\frac{1}{2}}\rho^{\abs{k_2}}\norm{v}^{-n}.
\end{align}
If $\abs{k_2}<4\max\{1,\gamma'^{-2}\}\norm{v}^{2(n+1)}$, then
\begin{align}\label{for:19}
  \norm{F^{k_1}Q^{k_2}v}&\geq C_F\rho^{\abs{k_1}}\min\{Q^{k_2}\norm{\pi_1(v)},Q^{k_2}\norm{\pi_3(v)}\}\notag\\
  &\geq C_{F,Q}\rho^{\abs{k_1}}\abs{k_2}^{-n}\min\{\norm{\pi_1(v)},\norm{\pi_3(v)}\}\notag\\
  &\geq C_{F,Q}\rho^{\abs{k_1}}\abs{k_2}^{-n}\norm{v}^{-n}\notag\\
  &\geq C_{F,Q}\rho^{\abs{k_1}}\norm{v}^{-n_1}.
\end{align}
Combine \eqref{for:20} and \eqref{for:19} we get the conclusion.
\medskip

\textbf{Proof of \eqref{for:21}}: The claim follows from the estimate in \eqref{for:17} and the fast decay
of Fourier coefficients:
\begin{align}
 \abs{S_K}&\leq C_{p_1}\norm{\varphi}_a\sum_{k=(k_1,k_2)\in K}\abs{y}^{k_1}\abs{k_2}^{\text{deg}(p_1)}\norm{Q^{k_1}F^{k_2}v}^{-a}\notag\\
 &\leq C_{a,p_1}\norm{\varphi}_a\sum_{k\in K}\abs{y}^{k_1}\rho^{-a\abs{k_1}}\abs{k_2}^{-\frac{1}{2}a+\text{deg}(p_1)}\norm{v}^{an_1}\label{for:5}.
\end{align}
The last sum clearly converges providing $a>\max\{\abs{\log_\rho \abs{y}},\,2(1+\text{deg}(p_1))\}$ and for a $C^\infty$
function $\varphi$ we can choose a as large as needed.

\smallskip
\textbf{Proof of \eqref{for:22}}: From estimate in \eqref{for:17} we can write
\begin{align}\label{for:24}
  \norm{F^{k_1}Q^{k_2}v}&\geq C\abs{k_2}^{\frac{1}{2}}\rho^{\abs{k_1}}\norm{v}^{-n_1}\notag\\
  &=\left\{\begin{aligned} &C\abs{k_2}^{\frac{1}{2}}\rho^{\abs{k_1}-\tau_0}(\rho^{\tau_0}\norm{v}^{-n_1})\notag\\
&C\abs{k_2k_0^{-1}}^{\frac{1}{2}}\rho^{\abs{k_1}}(\abs{k_0}^{\frac{1}{2}}\norm{v}^{-n_1})
\end{aligned}
 \right.\notag\\
 &\geq\left\{\begin{aligned} &C\abs{k_2}^{\frac{1}{2}}\rho^{\abs{k_1}-\tau_0}\norm{v}\quad \,\,(*)\\
&C\abs{k_2k_0^{-1}}^{\frac{1}{2}}\rho^{\abs{k_1}}\norm{v}\quad (**)
\end{aligned}
 \right.
  \end{align}
providing $\abs{k_1}\geq \tau_0=[(n_1+1)\log_\rho\norm{v}]+1$ or $\abs{k_2}\geq k_0=[\norm{v}^{2(n_1+1)}]+1$.
\begin{align*}
 \abs{S_K}\leq\sum_{\{k\in K:\abs{k_1}\geq \tau_0\}}+\sum_{\{k\in K:\abs{k_2}\geq k_0^2\}}+\sum_{\{k\in K:\abs{k_1}<\tau_0,\,\abs{k_2}<k_0^2\}}
\end{align*}
To estimate the first sum we use $(*)$ of \eqref{for:24}:
\begin{align*}
&\abs{S_{\{k\in K:\abs{k_1}\geq \tau_0\}}}\\
&\leq C_{p_1}\norm{\varphi}_a\sum_{\{k\in K:\abs{k_1}\geq \tau_0\}}\abs{y}^{k_1}\abs{k_2}^{\text{deg}(p_1)}\norm{F^{k_1}Q^{k_2}v}^{-a}\\
&\leq C_{a,p_1}\norm{\varphi}_a\sum_{\{k\in K:\abs{k_1}\geq \tau_0\}}\abs{y}^{k_1}\abs{k_2}^{-\frac{1}{2}a+\text{deg}(p_1)}\rho^{-a(\abs{k_1}-\tau_0)}\norm{v}^{-a}\\
&\leq C_{a,p_1}\max\{\abs{y},\abs{y}^{-1}\}^{\tau_0}\norm{\varphi}_a\norm{v}^{-a}\\
&\leq C_{a,\abs{y},\abs{y}^{-1}}\norm{\varphi}_a\norm{v}^{-a+\kappa}
\end{align*}
for any $a>\kappa=\max\{(n_1+1)\abs{\log_\rho \abs{y}},\,2(1+\text{deg}(p_1))\}$.

 To estimate the second sum we use $(**)$ of \eqref{for:24}:
\begin{align*}
&\abs{S_{\{k\in K:\abs{k_2}\geq k_0^2\}}}\\
&\leq C_{p_1}\norm{\varphi}_a\sum_{\{k\in K:\abs{k_2}\geq k_0^2\}}\abs{y}^{k_1}\abs{k_2}^{\text{deg}(p_1)}\norm{F^{k_1}Q^{k_2}v}^{-a}\\
&\leq C_{a,p_1}\norm{\varphi}_a\sum_{\{k\in K:\abs{k_2}\geq k_0^2\}}\abs{y}^{k_1}\rho^{-a\abs{k_1}}\abs{k_2}^{\text{deg}(p_1)}\abs{k_2k_0^{-1}}^{-\frac{1}{2}a}\norm{v}^{-a}\\
&\leq C_a\abs{k_2}^{-\frac{1}{4}a+\text{deg}(p_1)}\norm{\varphi}_a\norm{v}^{-a}\\
&\leq C_{a}\norm{\varphi}_a\norm{v}^{-a}
\end{align*}
for any $a>\max\{4(1+\text{deg}(p_1)),\,\abs{\log_\rho \abs{y}}\}$.

To estimate the third sum we use the additional assumption:
\begin{align*}
&\abs{S_{\{k\in K:k_1<\tau_0,\abs{k_2}\leq k_0^2\}}}\\
&\leq C_{p_1}\norm{\varphi}_a\sum_{\{k\in K:\abs{k_1}<\tau_0,\,\abs{k_2}\leq k_0^2\}}\abs{y}^{k_1}\abs{k_2}^{\text{deg}(p_1)}\norm{F^{k_1}Q^{k_2}v}^{-a}\\
&\leq C_{a,p_1}k_0^{2+2\text{deg}(p_1)}\max\{\abs{y},\,\abs{y}^{-1}\}^{\tau_0}\norm{\varphi}_a\sum_{\{k\in K:\abs{k_1}<\tau_0\}}p(\abs{k_1})^a\norm{v}^{-a}\\
&\leq C_ak_0^{2+2\text{deg}(p_1)}\tau_0\max\{\abs{y},\,\abs{y}^{-1}\}^{\tau_0}\tau_0^{a\text{deg}(p)}\norm{\varphi}_a\norm{v}^{-a}\\
&\leq C_{a,\delta}\norm{\varphi}_a\norm{v}^{-a+\kappa_1+\delta}
\end{align*}
for any $\delta>0$ and any
\begin{align*}
a>\kappa_1=(n_1+1)(4+4\text{deg}(p_1)+\abs{\log_\rho \abs{y}}).
\end{align*}
By combining the estimates obtained above we get the conclusion.
\end{proof}
The next result follows immediately from Lemma \ref{le:10}:
\begin{corollary}\label{cor:4}
Let $F$ and $Q$ be integer matrices in $GL(n,\ZZ)$. Suppose $F$ is ergodic and $Q$ is unipotent such that $FQ=QF$.
Then:
\begin{enumerate}

  \item \label{for:32}For any $C^{\infty}$ map $\varphi$ on the torus, any vector $v\in\ZZ^n$ satisfying $Qv\neq v$, any $y\in\RR$ and a polynomial $p_1$
the following sums:
\begin{align*}
  S_K(\varphi,v)(F,Q)=\sum_{k=(k_1,k_2)\in K}F^{-(k_1+1)}Q^{-(k_2+1)}\widehat{\varphi}_{F^{k_1}Q^{k_2}v}
\end{align*}
converge absolutely for any $K\subset\ZZ^2$.

\smallskip
  \item \label{for:25}Assume in addition that for any vector $v\in\ZZ^N$ and for
every $k=(k_1,k_2)\in K=K(v)\subset\ZZ^2$ we have $p(\abs{k_1})\norm{F^{k_1}Q^{k_2}v}\geq \norm{v}$ where $p$ is a polynomial, then
\begin{align*}
\norm{S_K(\varphi,v)(F,Q)}\leq C_{a}\norm{\varphi}_a\norm{v}^{-a+\kappa}
\end{align*}
for any $a>\kappa_{F,Q}=(n_1+1)(4+4n+\abs{\log_\rho \abs{y}})$, where $n_1=(2n+3)n$, $\rho>1$ is the growth rate in the hyperbolic direction corresponding to $F$ and
$y=\max\{\norm{F},\norm{F^{-1}}\}$.
\end{enumerate}
\end{corollary}
\begin{proof}
Since
\begin{align*}
  &\sum_{k=(k_1,k_2)\in K}\norm{F^{-(k_1+1)}Q^{-(k_2+1)}\widehat{\varphi}_{F^{\abs{k_1}+1}Q^{k_2}v}}\\
  &\leq C_{F,Q}\max\{\norm{F},\norm{F^{-1}}\}^{k_1}(\abs{k_2}+1)^{n}\sum_{k=(k_1,k_2)\in K}\norm{\widehat{\varphi}_{F^{k_1}Q^{k_2}v}}
\end{align*}
the conclusion follows immediately from Lemma \ref{le:10}. Here we used the fact that $Q$ has polynomial increasing speed, i.e.,
$\norm{Q^{k_2}}\leq C_Q(\abs{k_2}+1)^n$ for any $k_2\in\ZZ$.
\end{proof}

\subsection{Twisted coboundary equation of a map over automorphism on torus} Obstructions to solving a one-cohomology equation for a function over an ergodic toral automorphism
in $C^\infty$ category are sums of Fourier coefficients of the given function
along a dual orbit of the automorphism. This is the content of the Lemma 4.2 in \cite{Damjanovic4}. The same characterization holds however for
one-cohomology equation for a map over
ergodic toral automorphisms as well due to the estimate in Corollary \ref{cor:2}. The proofs of the two lemmas below follow closely the proof of Lemma 4.2 in \cite{Damjanovic4}
for solving a one-cohomology equation for functions. Details of the proofs can be found in \cite{wang}.

\begin{lemma}\label{le:8}
Let $P$ and $Q$ are integer matrices in $GL(N,\ZZ)$ and $Q$ is ergodic. For a map $\theta$ on $\TT^N$, if there exists a $C^{\infty}$  map $\omega$ which is $C^0$ small enough on $\TT^N$ such that
\begin{align}\label{for:18}
 P\omega-\omega\circ Q=\theta,
\end{align}
then  the following sums along all nonzero dual orbits are zero, i.e.,
\begin{align*}
\sum_{i=-\infty}^\infty P^{-(i+1)}\hat{\theta}_{Q^iv}=0,\qquad \forall v\neq0.
\end{align*}
\end{lemma}

\begin{lemma}\label{le:5}
 Let $P$ and $Q$ be ergodic integer matrices in $SL(N,\ZZ)$.  Let $\theta$ be a $C^\infty$ map on the torus which is $C^\sigma$ small enough, where $\sigma>N+2+\kappa_{P,Q}$ ($\kappa_{P,Q}$ is defined in \eqref{for:34} of Corollary \ref{cor:2}).  If for all nonzero $v\in\ZZ^N$, the following sums along the dual orbits are zero, i.e.,
\begin{align*}
\sum_{i=-\infty}^{\infty}P^{-(i+1)}\widehat{\theta}_{Q^iv}=0,\qquad \forall v\neq 0.
\end{align*}
Then the equation
\begin{align}\label{for:5}
  P\omega-\omega\circ Q=\theta
\end{align}
has a $C^\infty$ solution $\omega$, and the following estimate:
\begin{align*}
  \norm{\omega}_{a}\leq C_a\norm{\theta}_{a+\sigma_1},\qquad \forall a\geq 0.
\end{align*}
where $\sigma_1>\kappa_{P,Q}$.
\end{lemma}

\section{Construction of the projection for action $\alpha$ when $\mathcal{H}$ is nilpotent}\label{sec:2}
Set $c=\max\{\norm{B}, \norm{B^{-1}}\}$ and $l>[\max\{N_1(c), N_2(c)\}]+1$ where $N_1(c)$ and $N_2(c)$ are defined in Lemma \ref{le:9} and Corollary \ref{cor:3} respectively. Let $\mathcal{B}=B$ and $\mathcal{A}=A^{l}$. We will use ergodic element $\mathcal{A}$ instead of $A$ to carry out KAM scheme.

The crucial step in proving Theorem \ref{th:1} is Proposition \ref{po:2}. The basic idea is as follows: we can make a reduction to consider $R_{g_1}$ map ($\mathcal{A}=\alpha(g_1)$) with $(\widehat{R_{g_1}})_v=0$ if $v\notin E_\mathcal{A}$. Once only the $E_\mathcal{A}$-Fourier coefficients are mattered for $R_{g_1}$,  $(\widehat{R_{d_i}})_v$ ($0\leq i\leq\mathfrak{n}$) (see Step $3$ of Section \ref{sec:1}) are very small if $v$ are ``far from" $E_\mathcal{A}$. Hence we can center on finite points to compute the obstruction.

In the proof of next proposition, Lemma \ref{cor:3} will be used frequently to simplify computation.

\begin{proposition}\label{po:2}
For the cocycle difference equation \eqref{for:57}, set $\norm{\mathcal{L}}_a=\max_{0\leq i\leq\mathfrak{n}}\{\norm{\mathcal{L}(g_1,d_i)}_a\}$ (see Step $3$ of Section \ref{sec:1}). If $v\in E_{\mathcal{A}}(v)$, then
\begin{align}\label{for:66}
 &\Big|\sum_j \mathcal{B}\mathcal{A}^{-(j+1)}(\widehat{R_{g_1}})_{\mathcal{A}^jv}-\sum_j \mathcal{A}^{-(j+1)}(\widehat{R_{g_1}})_{\mathcal{A}^j\mathcal{B}v}\Big|\notag\\
 &\leq C_a\norm{\mathcal{L}}_a\abs{v}^{-a+(\mathfrak{n}+1)\sigma}.
\end{align}
where $\sigma>\max_{0\leq i\leq\mathfrak{n}}\{\kappa_{\mathcal{A},\mathcal{A}\overline{d_i}}\}$ ($\kappa_{\mathcal{A},\mathcal{A}\overline{d_i}}$ is as defined in Corollary \ref{cor:2}) and $a>(\mathfrak{n}+1)\sigma$

\end{proposition}
\begin{proof}
The following facts will be used frequently in the proof:
\begin{enumerate}
\item condition \eqref{for:9} is satisfied for any $\mathcal{A}\overline{d_i}$ and $v\in\ZZ^N\backslash0$ on $K=\{(k_1,0),k_1\in\ZZ\}$, $0\leq i\leq\mathfrak{n}$ (see \eqref{cor:5} and Remark \ref{re:1}).

\noindent

\smallskip
  \item $\sum_j\norm{\mathcal{A}^j\widehat{f}_{(\mathcal{A}\overline{d_i})^jv}}<\infty$ ($0\leq i\leq\mathfrak{n}$) where $f$ is a smooth map and $v\in\ZZ^N\backslash0$ (see Corollary \ref{cor:2}).

  \smallskip
  \item \label{for:45}if $\mathcal{A}^nv\in E_\mathcal{A}$, $n\geq 0$ (resp. $n\leq0$), then on $K^-=\{(k_1,0)\in\ZZ:k_1\leq 0\}$ (resp. $K^+=\{(k_1,0)\in\ZZ:k_1\geq0\}$) the condition \eqref{for:96} is satisfied for any $\mathcal{A}\overline{d_i}$ and $v$ on $K^-$ (resp. $K^+$) (see Proposition \ref{po:1}).

      \smallskip
      \item Using $\mathcal{A}\overline{d_i}=\overline{d_i}\mathcal{A}\overline{d_{i+1}}$, $0\leq i\leq \mathfrak{n}$, where $d_{\mathfrak{n}+1}=e$, equation \eqref{for:57} has following forms for the pairs $(\mathcal{A},\,\overline{d_i})$:
\begin{align}
  &\mathcal{A}R_{d_i}-R_{d_i}\circ \mathcal{A}\overline{d_{i+1}}\notag\\
  &=\overline{d_i}R_{g_1}\circ \overline{d_{i+1}}-R_{g_1}\circ \overline{d_i}+\overline{d_i}AR_{d_{i+1}}+\mathcal{L}(g_1,d_i); \label{for:60}
  \end{align}
 for $0\leq i\leq \mathfrak{n}$, where $R_{d_{\mathfrak{n}+1}}=0$. We will be focus on these relations.
 \end{enumerate}
Next, we proceed to the proof.

\noindent\textbf{Step 1}: Reduction to prove for the``reminder" maps.

Let $\mathcal{R}R_{g_1}=\sum_{u\in\ZZ^N}(\widehat{\mathcal{R}R_{g_1}})_ue_u$ where
\begin{align}\label{for:3}
(\widehat{\mathcal{R}R_{g_1}})_u&\stackrel{\rm def}{=}\left\{\begin{aligned} &\sum_{i\in\ZZ}\mathcal{A}^{-i}(\widehat{\mathcal{R}R_{g_1}})_{\mathcal{A}^iu},\qquad &u\in E_{\mathcal{A}},\\
&0, \qquad &\text{otherwise}
\end{aligned}
 \right.
\end{align}
for $u\neq 0$ and $(\widehat{\mathcal{R}R_{g_1}})_0\stackrel{\rm def}{=}0$.

Application of \eqref{for:34} of Corollary \ref{cor:2} shows that
\begin{align}\label{for:31}
 \norm{\mathcal{R}R_{g_1}}_{a}\leq C_a\norm{R_{g_1}}_{a+\sigma_1},\qquad \forall a\geq 0.
\end{align}
where $\sigma_1>\kappa_{\mathcal{A},\mathcal{A}}$.

Since $R_{g_1}-\mathcal{R}R_{g_1}$ satisfies the solvable condition in Lemma \ref{le:5}, by using Lemma \ref{le:5} there is a $C^\infty$ function $\Omega$ such that
\begin{align}\label{for:30}
 \Delta_\mathcal{A}\Omega=R_{g_1}-\mathcal{R}R_{g_1}
\end{align}
(see \eqref{for:46} of Section \ref{sec:1}) with estimates
\begin{align}\label{for:16}
 \norm{\Omega}_{a}\leq C_a\norm{R_{g_1}-\mathcal{R}R_{g_1}}_{a+\sigma_1}\leq C_a\norm{R_{g_1}}_{a+2\sigma_1}, \quad\forall a\geq 0.
\end{align}
Let
\begin{align*}
 \mathcal{R}R_{x}=R_{x}-\Delta_{\bar{x}}\Omega,\qquad \forall x\in\mathcal{H}.
\end{align*}
It is easy to check that if we substitute $R_x$ by $\mathcal{R}R_{x}$, for any $x\in\mathcal{H}$
then equation \eqref{for:57} is also satisfied by these ``remainder" maps. This shows that we can just prove the conclusion for assuming that $R_{g_1}$ satisfies the condition:
\begin{align}\label{for:56}
 (\widehat{R_{g_1}})_u=0,\qquad \text{if } u\notin E_\mathcal{A}.
\end{align}

\smallskip
\noindent{\textbf{Step $2$}: \emph{Reduction to  maps concentrated near $E_\mathcal{A}$ }}. In this part we want to show: if $\mathcal{A}^nu\in E_\mathcal{A}$ for $n\geq 1$ or $n\leq-2$, then
\begin{align}\label{for:72}
 \abs{(\widehat{R_{d_{i}}})_u}\leq C_a\norm{\mathcal{L}}_a\norm{u}^{-a+(\mathfrak{n}+1)\sigma},
\end{align}
for any $a>(\mathfrak{n}+1)\sigma$, $i\leq \mathfrak{n}$.

We define $\varphi_{i}=\sum_{u\in\ZZ^N}(\widehat{\varphi_{i}})_ue_u$, $0\leq i\leq \mathfrak{n}$:
\begin{align}\label{for:62}
(\widehat{\varphi_{i}})_u&\stackrel{\rm def}{=}\left\{\begin{aligned} &-\sum_{j\leq-1}\mathcal{A}^{-(j+1)}(\widehat{\mathcal{L}_i})_{(\mathcal{A}\overline{d_{i+1}})^ju},\quad &\mathcal{A}^nu\in E_\mathcal{A}, n&\geq 1,\\
&\sum_{j\geq 0}\mathcal{A}^{-(j+1)}(\widehat{\mathcal{L}_i})_{(\mathcal{A}\overline{d_{i+1}})^ju},\quad &\mathcal{A}^nu\in E_\mathcal{A}, n&\leq-2,\\
&0, \quad &\text{otherwise}
\end{aligned}
 \right.
\end{align}
where
\begin{align}\label{for:47}
 \mathcal{L}_i=\mathcal{L}(g_1,d_{i})+\overline{d_i}\mathcal{A}R_{d_{i+1}}.
\end{align}
\eqref{for:45} shows that $\varphi_i$,  $0\leq i\leq \mathfrak{n}$ are $C^\infty$ maps. Furthermore, $R_{\mathfrak{n}+1}=0$ implies:
\begin{align}\label{for:71}
 \norm{\varphi_{\mathfrak{n}}}_{a}\leq C_a\norm{\mathcal{L}(g_1,d_{\mathfrak{n}})}_{a+\sigma_1},\qquad \forall a\geq 0.
\end{align}
For each $0\leq i\leq \mathfrak{n}$, by iterating \eqref{for:60} backwards with respect to $\mathcal{A}\overline{d_{i+1}}$ we get
\begin{align}
  &-\sum_{j\leq-1}\mathcal{A}^{-(j+1)}\overline{d_{i}}(\widehat{R_{g_1}})_{\overline{d_{i+1}}(\mathcal{A}\overline{d_{i+1}})^ju}
  +\sum_{j\leq-1}\mathcal{A}^{-(j+1)}(\widehat{R_{g_1}})_{\overline{d_{i}}(\mathcal{A}\overline{d_{i+1}})^ju}\notag\\
  &=(\widehat{R_{d_{i}}})_u+\sum_{j\leq-1}\mathcal{A}^{-(j+1)}(\widehat{\mathcal{L}_i})_{(\mathcal{A}\overline{d_{i+1}})^ju}\label{for:48}\\
  &=(\widehat{R_{d_{i}}})_u-(\widehat{\varphi_{i}})_u\notag.
  \end{align}
By Lemma \ref{cor:3} and \eqref{for:56} we get
\begin{align*}
  (\widehat{R_{d_{i}}})_u=(\widehat{\varphi_{i}})_u,\qquad \text{if }\mathcal{A}^nu\in E_\mathcal{A},\,n\geq 1.
\end{align*}
By iterating \eqref{for:60} with respect to $\mathcal{A}\overline{d_{i+1}}$ we get
\begin{align}\label{for:52}
  &\sum_{j\geq0}\mathcal{A}^{-(j+1)}\overline{d_{i}}(\widehat{R_{g_1}})_{\overline{d_{i+1}}(\mathcal{A}\overline{d_{i+1}})^ju}
  -\sum_{j\geq0}\mathcal{A}^{-(j+1)}(\widehat{R_{g_1}})_{\overline{d_{i}}(\mathcal{A}\overline{d_{i+1}})^ju}\notag\\
  &=(\widehat{R_{d_{i}}})_u-\sum_{j\geq0}\mathcal{A}^{-(j+1)}(\widehat{\mathcal{L}_i})_{(\mathcal{A}\overline{d_{i+1}})^ju}\\
  &=(\widehat{R_{d_{i}}})_u-(\widehat{\varphi_{i}})_u\notag.
  \end{align}
By Lemma \ref{cor:3} and \eqref{for:56}, we get
\begin{align*}
  (\widehat{R_{d_{i}}})_u=(\widehat{\varphi_{i}})_u,\qquad \text{if }\mathcal{A}^nu\in E_\mathcal{A},\,n\leq-2.
\end{align*}
Hence,
\begin{align}\label{for:65}
  (\widehat{R_{d_{i}}})_u=(\widehat{\varphi_{i}})_u,\qquad \text{if }\mathcal{A}^nv\in E_\mathcal{A},\,n\geq1\text{ or }n\leq-2.
\end{align}
This shows that
\begin{align}\label{for:112}
  (\widehat{\mathcal{L}_i})_u&=\widehat{\mathcal{L}(g_1,d_{i})}_u+\overline{d_i}\mathcal{A}(\widehat{R_{d_{i+1}}})_u
  =\widehat{\mathcal{L}(g_1,d_{i})}_u+\overline{d_i}\mathcal{A}(\widehat{\varphi_{i+1}})_u,
\end{align}
if $\mathcal{A}^nu\in E_\mathcal{A}$, $n\geq 1$ or $n\leq-2$.

Hence by using Corollary \ref{cor:2} it follows from \eqref{for:62} and \eqref{for:112} that
\begin{align*}
 \norm{\varphi_{i}}_{a}&\leq C_a\norm{\mathcal{L}(g_1,d_{i})+d_i\mathcal{A}\varphi_{i+1}}_{a+\sigma}\notag\\
 &\leq C_a(\norm{\mathcal{L}}_{a+\sigma}+\norm{\varphi_{i+1}}_{a+\sigma})\qquad \forall a\geq 0.
\end{align*}
This and \eqref{for:71} imply that
\begin{align}\label{for:113}
\norm{\varphi_{i}}_{a}\leq C_a\norm{\mathcal{L}}_{a+(\mathfrak{n}-i+1)\sigma},\qquad \forall a\geq 0,
\end{align}
for any $0\leq i\leq \mathfrak{n}$. Hence we proved \eqref{for:72}.

\smallskip
\noindent{\textbf{Step $3$}: \emph{Basic properties of $R_{d_i},\,0\leq i\leq\mathfrak{n}$ }}. If $v\in E_{\mathcal{A}}$, then in \eqref{for:48} substituting $u$ by $v$ and using Lemma \ref{cor:3}, \eqref{for:56} and \eqref{for:65} we have
\begin{align}\label{for:76}
 (\widehat{R_{d_i}})_v&=(\widehat{R_{g_1}})_{\mathcal{A}^{-1}\overline{d_{i}}v}-\overline{d_i}\mathcal{A}(\widehat{R_{d_{i+1}}})_{(\mathcal{A}\overline{d_{i+1}})^{-1}v}\notag\\
 &-\sum_{j\leq-1}\mathcal{A}^{-(j+1)}(\widehat{\mathcal{L}(g_1,d_i)})_{(\mathcal{A}\overline{d_{i+1}})^{j}v}\notag\\
 &-\sum_{j\leq-2}\mathcal{A}^{-(j+1)}\overline{d_i}\mathcal{A}(\widehat{\varphi_{i+1}})_{(\mathcal{A}\overline{d_{i+1}})^{j}v}.
\end{align}
Here we used the relation $\overline{d_{i}}(\mathcal{A}\overline{d_{i+1}})^{-1}=\mathcal{A}^{-1}\overline{d_{i}}$ and \eqref{for:47}.

In \eqref{for:48} substituting $u$ by $(\mathcal{A}\overline{d_{m}})^{-1}v$ and using Lemma \ref{cor:3}, \eqref{for:56} and \eqref{for:65} we have
\begin{align*}
 (\widehat{R_{d_{i}}})_{(\mathcal{A}\overline{d_{m}})^{-1}v}&=-\sum_{j\leq-1}A^{-(j+1)}\overline{d_{i}}\mathcal{A}
 (\widehat{\varphi_{i+1}})_{(\mathcal{A}\overline{d_{i+1}})^{j}(\mathcal{A}\overline{d_{m}})^{-1}v}\\
 &-\sum_{j\leq-1}\mathcal{A}^{-(j+1)}
 (\widehat{\mathcal{L}(g_1,d_{i})})_{(\mathcal{A}\overline{d_{i+1}})^{j}(\mathcal{A}\overline{d_{m}})^{-1}v}.
\end{align*}
Lemma \ref{cor:3} and \eqref{for:45} show that we can use Corollary \ref{cor:2} to estimate the sums:
\begin{align}\label{for:75}
 \big\|(\widehat{R_{d_{i}}})_{(\mathcal{A}\overline{d_{m}})^{-1}v}\big\|&\leq C_a(\norm{\varphi_{i+1}}_a+\norm{\mathcal{L}}_a)\norm{\mathcal{A}\overline{d_{m}})^{-1}v}^{-a+\sigma}\notag\\
 &\leq C_a\norm{\mathcal{L}}_{a+(\mathfrak{n}+1)\sigma}\norm{v}^{-a+\sigma}
\end{align}
for any $v\in E_\mathcal{A}$ and $0\leq m,\,i\leq\mathfrak{n}$. Here we used estimate \eqref{for:113} for $\varphi_{i+1}$.

Applying \eqref{for:75} and Corollary \ref{cor:2} to \eqref{for:76}, we get
\begin{align}\label{for:80}
  &\big\|(\widehat{R_{d_i}})_v-(\widehat{R_{g_1}})_{\mathcal{A}^{-1}\overline{d_{i}}v}\big\|\notag\\
  &\leq \|\overline{d_i}\mathcal{A}(\widehat{R_{d_{i+1}}})_{(\mathcal{A}\overline{d_{i+1}})^{-1}v}\|+C_a(\norm{\varphi_{i+1}}_a
  +\norm{\mathcal{L}_a})\norm{v}^{-a+\sigma}\notag\\
  &\leq C_a\norm{\mathcal{L}}_{a+(\mathfrak{n}+1)\sigma}\norm{v}^{-a+\sigma}.
\end{align}
In \eqref{for:52} substituting $u$ by substitute $u$ by $\mathcal{A}\overline{d_{m}}\mathcal{A}\overline{d_{n}}v$ and using Lemma \ref{cor:3}, \eqref{for:56} and \eqref{for:65} we have
\begin{align*}
 (\widehat{R_{d_i}})_{\mathcal{A}\overline{d_{m}}\mathcal{A}\overline{d_{n}}v}&=\mathcal{A}^{-1}\overline{d_i}\mathcal{A}
 (\widehat{R_{d_{i+1}}})_{\mathcal{A}\overline{d_{m}}\mathcal{A}\overline{d_{n}}v}\\
 &+\sum_{j\geq0}\mathcal{A}^{-(j+1)}(\widehat{\mathcal{L}(g_1,d_i)})_{(\mathcal{A}\overline{d_{i+1}})^{j}\mathcal{A}\overline{d_{m}}\mathcal{A}\overline{d_{n}}v}.
\end{align*}
\eqref{for:45} shows that we can use Corollary \ref{cor:2} to estimate difference:
\begin{align*}
 &\big|(\widehat{R_{d_i}})_{\mathcal{A}\overline{d_{m}}\mathcal{A}\overline{d_{n}}v}-\mathcal{A}^{-1}
 \overline{d_i}\mathcal{A}(\widehat{R_{d_{i+1}}})_{\mathcal{A}\overline{d_{m}}\mathcal{A}\overline{d_{n}}v}\big|\\
 &\leq C_a\norm{\mathcal{L}}_a\norm{\mathcal{A}\overline{d_{m}}\mathcal{A}\overline{d_{n}}v}^{-a+\sigma}\\
 &\leq C_a\norm{\mathcal{L}}_a\norm{v}^{-a+\sigma}
\end{align*}
for any $a>\sigma$.

Note that $R_{d_{\mathfrak{n}+1}}=0$, then above inequality implies
\begin{align}\label{for:77}
 &\big|(\widehat{R_{d_i}})_{\mathcal{A}\overline{d_{m}}\mathcal{A}\overline{d_{n}}v}\big|\leq C_a\norm{\mathcal{L}}_a\norm{v}^{-a+\sigma},
\end{align}
for any $a>\sigma$, if $v\in E_\mathcal{A}$ and $0\leq i,m,n\leq\mathfrak{n}$.

\smallskip
\noindent{\textbf{Step $4$}: \emph{Proof of the result }}

In \eqref{for:60}, for any $0\leq i\leq \mathfrak{n}$ Lemma \ref{le:8} shows that  the obstructions for
\begin{align*}
 \overline{d_i}R_{g_1}\circ \overline{d_{i+1}}-R_{g_1}\circ \overline{d_i}+\overline{d_i}\mathcal{A}R_{d_{i+1}}+\mathcal{L}(g_1,d_i)
\end{align*}
with respect to $\mathcal{A}\overline{d_{i+1}}$ vanish; therefore we get
\begin{align}\label{for:84}
&\sum_j\mathcal{A}^{-(j+1)}(\widehat{R_{g_1}})_{\mathcal{A}^j\overline{d_{i}}v}\stackrel{(1)}{=}\sum_j\mathcal{A}^{-(j+1)}
(\widehat{R_{g_1}})_{\overline{d_{i}}(\mathcal{A}\overline{d_{i+1}})^jv}\notag\\
 &=\sum_j\mathcal{A}^{-(j+1)}\overline{d_{i}}(\widehat{R_{g_1}})_{\overline{d_{i+1}}(\mathcal{A}\overline{d_{i+1}})^jv}\notag\\
 &+\sum_j\mathcal{A}^{-(j+1)}\overline{d_i}\mathcal{A}(\widehat{R_{d_{i+1}}})_{(\mathcal{A}\overline{d_{i+1}})^jv}\notag\\
 &+\sum_j\mathcal{A}^{-(j+1)}(\widehat{\mathcal{L}(g_1,d_{i})})_{(\mathcal{A}\overline{d_{i+1}})^jv}\notag\\
 &\stackrel{(2)}{=}\overline{d_{i}}\sum_j(\mathcal{A}\overline{d_{i+1}})^{-(j+1)}(\widehat{R_{g_1}})_{\overline{d_{i+1}}(\mathcal{A}\overline{d_{i+1}})^jv}\notag\\
 &+\overline{d_i}\sum_j(\mathcal{A}\overline{d_{i+1}})^{-(j+1)}\mathcal{A}(\widehat{R_{d_{i+1}}})_{(\mathcal{A}\overline{d_{i+1}})^jv}\notag\\
 &+\sum_j\mathcal{A}^{-(j+1)}(\widehat{\mathcal{L}(g_1,d_{i})})_{(\mathcal{A}\overline{d_{i+1}})^jv}.
\end{align}
Here in $(1)$ and $(2)$ we use the relation $\mathcal{A}^j\overline{d_{i}}=\overline{d_{i}}(\mathcal{A}\overline{d_{i+1}})^j$ for any $j\in\ZZ$.

Especially, for $i=\mathfrak{n}$ we have
\begin{align*}
 &\sum_j\overline{d_{\mathfrak{n}}^{-1}}\mathcal{A}^{-(j+1)}(\widehat{R_{g_1}})_{\mathcal{A}^j\overline{d_{\mathfrak{n}}}v}-\sum_j\mathcal{A}^{-(j+1)}
 (\widehat{R_{g_1}})_{\mathcal{A}^jv}\\
 &=\sum_j\overline{d_{\mathfrak{n}}^{-1}}\mathcal{A}^{-(j+1)}(\widehat{\mathcal{L}(g_1,d_{\mathfrak{n}})})_{\mathcal{A}^jv}.
\end{align*}
Here we used $d_{\mathfrak{n}+1}=e$ and $R_{\mathfrak{n}+1}=0$.

Then by Corollary \ref{cor:2} we obtain
\begin{align}\label{for:82}
 &\Big|\overline{d_{\mathfrak{n}}^{-1}}\sum_j\mathcal{A}^{-(j+1)}(\widehat{R_{g_1}})_{A^j\overline{d_{\mathfrak{n}}}v}
 -\sum_j\mathcal{A}^{-(j+1)}(\widehat{R_{g_1}})_{\mathcal{A}^jv}\Big|\notag\\
 &\leq C_a\norm{\mathcal{L}}_a\norm{v}^{-a+\sigma}
\end{align}
for any $a>\sigma$.

Hence if we can prove:
\begin{align}\label{for:83}
 &\Big|\overline{d_{i}^{-1}}\sum_j\mathcal{A}^{-(j+1)}(\widehat{R_{g_1}})_{\mathcal{A}^j\overline{d_{i}}v}-
 \overline{d_{i+1}^{-1}}\sum_j\mathcal{A}^{-(j+1)}(\widehat{R_{g_1}})_{\mathcal{A}^j\overline{d_{i+1}}v}\Big|\notag\\
 &\leq C_a\norm{\mathcal{L}}_a\norm{v}^{-a+(\mathfrak{n}+1)\sigma},
\end{align}
for each  $0\leq i\leq\mathfrak{n}-1$, then \eqref{for:66} follows from \eqref{for:82} and \eqref{for:83} immediately.

In \eqref{for:84} by using Lemma \ref{cor:3}, \eqref{for:56} and \eqref{for:65} we have
\begin{align*}
  &\sum_j\overline{d_{i}}^{-1}\mathcal{A}^{-(j+1)}(\widehat{R_{g_1}})_{\mathcal{A}^j\overline{d_{i}}v}\\
  &=\sum_{j=0,1}(\mathcal{A}\overline{d_{i+1}})^{-(j+1)}(\widehat{R_{g_1}})_{\overline{d_{i+1}}(\mathcal{A}\overline{d_{i+1}})^jv}\\
  &+\sum_{-1\leq j\leq1}(\mathcal{A}\overline{d_{i+1}})^{-(j+1)}\mathcal{A}(\widehat{R_{d_{i+1}}})_{(\mathcal{A}\overline{d_{i+1}})^jv}\\
  &+\sum_{j\leq-2,\,j\geq3}(\mathcal{A}\overline{d_{i+1}})^{-(j+1)}\mathcal{A}(\widehat{\varphi_{i+1}})_{(\mathcal{A}\overline{d_{i+1}})^jv}\\
  &+\sum_j\overline{d_{i}}^{-1}\mathcal{A}^{-(j+1)}(\widehat{\mathcal{L}(g_1,d_{i})})_{(\mathcal{A}\overline{d_{i+1}})^jv}.
\end{align*}
By using Lemma \ref{cor:3} and \eqref{for:56} we also have
\begin{align*}
  \overline{d_{i+1}^{-1}}\sum_j\mathcal{A}^{-(j+1)}(\widehat{R_{g_1}})_{\mathcal{A}^j\overline{d_{i+1}}v}
  =\overline{d_{i+1}^{-1}}\sum_{-1\leq j\leq1}\mathcal{A}^{-(j+1)}(\widehat{R_{g_1}})_{\mathcal{A}^j\overline{d_{i+1}}v}
\end{align*}
Set
\begin{align}\label{for:88}
  J_i(v)&=\sum_{j=0,1}(\mathcal{A}\overline{d_{i+1}})^{-(j+1)}(\widehat{R_{g_1}})_{\overline{d_{i+1}}(\mathcal{A}\overline{d_{i+1}})^jv}\notag\\
  &+\sum_{j=0,1}(\mathcal{A}\overline{d_{i+1}})^{-(j+1)}\mathcal{A}(\widehat{R_{d_{i+1}}})_{(\mathcal{A}\overline{d_{i+1}})^jv}\notag\\
&-\overline{d_{i+1}^{-1}}\sum_{-1\leq j\leq1}\mathcal{A}^{-(j+1)}(\widehat{R_{g_1}})_{\mathcal{A}^j\overline{d_{i+1}}v}\notag\\
&=\overline{d_{i+1}^{-1}}(\widehat{R_{d_{i+1}}})_{v}+(\mathcal{A}\overline{d_{i+1}})^{-2}\mathcal{A}(\widehat{R_{d_{i+1}}})_{\mathcal{A}\overline{d_{i+1}}v}\notag\\
&+(\mathcal{A}\overline{d_{i+1}})^{-2}(\widehat{R_{g_1}})_{\overline{d_{i+1}}\mathcal{A}\overline{d_{i+1}}v}
-\overline{d_{i+1}^{-1}}(\widehat{R_{g_1}})_{\mathcal{A}^{-1}\overline{d_{i+1}}v}\notag\\
&-\overline{d_{i+1}^{-1}}\mathcal{A}^{-2}(\widehat{R_{g_1}})_{\mathcal{A}\overline{d_{i+1}}v}.
\end{align}
Then it follows from \eqref{for:113}, \eqref{for:75} and Corollary \ref{cor:2} that
\begin{align}\label{for:78}
 &\Big|\overline{d_{i}^{-1}}\sum_j\mathcal{A}^{-(j+1)}(\widehat{R_{g_1}})_{\mathcal{A}^j\overline{d_{i}}v}-
 \overline{d_{i+1}^{-1}}\sum_j\mathcal{A}^{-(j+1)}(\widehat{R_{g_1}})_{\mathcal{A}^j\overline{d_{i+1}}v}\Big|\notag\\
 &\leq |J_i(v)|+\sum_{j\leq-2,\,j\geq3}\norm{(\mathcal{A}\overline{d_{i+1}})^{-(j+1)}\mathcal{A}(\widehat{\varphi_{i+1}})_{(\mathcal{A}\overline{d_{i+1}})^jv}}\notag\\
 &+\sum_j\norm{\overline{d_{i}}^{-1}\mathcal{A}^{-(j+1)}(\widehat{\mathcal{L}(g_1,d_{i})})_{(\mathcal{A}\overline{d_{i+1}})^{-1}v}}
 +\norm{\mathcal{A}(\widehat{R_{d_{i+1}}})_{(\mathcal{A}\overline{d_{i+1}})^jv}}\notag\\
 &\leq |J_i(v)|+C_a\norm{\mathcal{L}}_a\norm{v}^{-a+(\mathfrak{n}+2)\sigma}.
\end{align}
Then to prove \eqref{for:83} we need to estimate $|J_i(v)|$.

\eqref{for:80} has provided enough information for $(\widehat{R_{d_{i+1}}})_{v}$. Next, we will
center on the computation of $(\widehat{R_{d_{i+1}}})_{\mathcal{A}\overline{d_{i+1}}v}$.

For any $0\leq m\leq\mathfrak{n}$, in \eqref{for:52} substituting $i$ by $m$ and $u$ by $\mathcal{A}\overline{d_{i+1}}v$ and using Lemma \ref{cor:3}, \eqref{for:56} and \eqref{for:65} we get
\begin{align}\label{for:6}
 (\widehat{R_{d_{m}}})_{\mathcal{A}\overline{d_{i+1}}v}&=\mathcal{A}^{-1}\overline{d_{m}}(\widehat{R_{g_1}})_{\overline{d_{m+1}}\mathcal{A}\overline{d_{i+1}}v}
 -\mathcal{A}^{-1}(\widehat{R_{g_1}})_{\overline{d_{m}}\mathcal{A}\overline{d_{i+1}}v}\notag\\
 &+\mathcal{A}^{-1}\overline{d_m}\mathcal{A}(\widehat{R_{d_{m+1}}})_{\mathcal{A}\overline{d_{i+1}}v}
 +\Theta(m,i,v),
\end{align}
where
\begin{align}
 \Theta(m,i,v)&=\mathcal{A}^{-2}\overline{d_m}\mathcal{A}(\widehat{R_{d_{m+1}}})_{\mathcal{A}\overline{d_{m+1}}\mathcal{A}\overline{d_{i+1}}v}\notag\\
 &+\sum_{j\geq2}\mathcal{A}^{-(j+1)}\overline{d_m}A(\widehat{\varphi_{i+1}})_{(\mathcal{A}\overline{d_{m+1}})^{j}\mathcal{A}\overline{d_{i+1}}v}\notag\\
 &+\sum_{j\geq0}\mathcal{A}^{-(j+1)}(\widehat{\mathcal{L}(g_1,d_{m})})_{(\mathcal{A}\overline{d_{m+1}})^{j}\mathcal{A}\overline{d_{i+1}}v},
\end{align}
and
\begin{align}\label{for:79}
 \abs{\Theta(m,i,v)}\leq C_a\norm{\mathcal{L}}_a\norm{v}^{-a+(\mathfrak{n}+1)\sigma},\quad \forall a>(\mathfrak{n}+1)\sigma
\end{align}
follows from \eqref{for:113}, \eqref{for:77} and Corollary \ref{cor:3}.

Set $\Lambda(m,i,v)=(\widehat{R_{d_{m}}})_{\mathcal{A}\overline{d_{i+1}}v}-\mathcal{A}^{-1}\overline{d_m}\mathcal{A}(\widehat{R_{d_{m+1}}})_{\mathcal{A}\overline{d_{i+1}}v}.$
Then it follows that
\begin{align}\label{for:74}
 (\widehat{R_{d_{i+1}}})_{\mathcal{A}\overline{d_{i+1}}v}&=\Lambda(i+1,i,v)
 +\sum_{m=i+2}^{\mathfrak{n}}\mathcal{A}^{-1}(\prod_{j=i+2}^m\overline{d_{j-1}})\mathcal{A}\Lambda(m,i,v)
 \end{align}
by noting used $R_{d_{\mathfrak{n}+1}}=0$, where $\prod_{j=i+2}^m\overline{d_{j-1}}$ is defined in \eqref{sec:1} of Section \ref{sec:2}; and we also get
\begin{align}\label{for:7}
  \Lambda(m,i,v)&=\mathcal{A}^{-1}\overline{d_{m}}(\widehat{R_{g_1}})_{\overline{d_{m+1}}\mathcal{A}\overline{d_{i+1}}v}
 -\mathcal{A}^{-1}(\widehat{R_{g_1}})_{\overline{d_{m}}\mathcal{A}\overline{d_{i+1}}v}\notag\\
 &+\Theta(m,i,v).
\end{align}
This shows that we can express $\sum_{m=i+2}^{\mathfrak{n}}\mathcal{A}^{-1}(\prod_{j=i+2}^m\overline{d_{j-1}})\mathcal{A}\Lambda(m,i,v)$ as:
\begin{align*}
  &\sum_{m=i+2}^{\mathfrak{n}}\mathcal{A}^{-1}(\prod_{j=i+2}^m\overline{d_{j-1}})\mathcal{A}\Lambda(m,i,v)\notag\\
 &=\sum_{m=i+2}^{\mathfrak{n}}\mathcal{A}^{-1}(\prod_{j=i+2}^m\overline{d_{j-1}})\mathcal{A}\big(\mathcal{A}^{-1}
 \overline{d_{m}}(\widehat{R_{g_1}})_{\overline{d_{m+1}}\mathcal{A}\overline{d_{i+1}}v}\notag\\
 &-\mathcal{A}^{-1}(\widehat{R_{g_1}})_{\overline{d_{m}}\mathcal{A}\overline{d_{i+1}}v}\big)
+\sum_{m=i+2}^{\mathfrak{n}}\mathcal{A}^{-1}(\prod_{j=i+2}^m\overline{d_{j-1}})\mathcal{A}\Theta(m,i,v).
\end{align*}
The first sum can be simplified as:
\begin{align*}
 d_{i+1}\mathcal{A}^{-1}(\widehat{R_{g_1}})_{\mathcal{A}\overline{d_{i+1}}v}
 -\mathcal{A}^{-1}\overline{d_{i+1}}(\widehat{R_{g_1}})_{\overline{d_{i+2}}\mathcal{A}\overline{d_{i+1}}v}.
\end{align*}
Here we used $\mathcal{A}^{-1}(\prod_{j=i+2}^{\mathfrak{n}+1}\overline{d_{j-1}})=d_{i+1}\mathcal{A}^{-1}$. Then it follows from \eqref{for:74} and above analysis that
\begin{align*}
 (\widehat{R_{d_{i+1}}})_{\mathcal{A}\overline{d_{i+1}}v}&=\mathcal{A}^{-1}\overline{d_{i+1}}(\widehat{R_{g_1}})_{\overline{d_{i+2}}\mathcal{A}\overline{d_{i+1}}v}
 -\mathcal{A}^{-1}(\widehat{R_{g_1}})_{\overline{d_{i+1}}\mathcal{A}\overline{d_{i+1}}v}\\
 &+d_{i+1}\mathcal{A}^{-1}(\widehat{R_{g_1}})_{\mathcal{A}\overline{d_{i+1}}v}
 -\mathcal{A}^{-1}\overline{d_{i+1}}(\widehat{R_{g_1}})_{\overline{d_{i+2}}\mathcal{A}\overline{d_{i+1}}v}\\
&+\sum_{m=i+2}^{\mathfrak{n}}\mathcal{A}^{-1}(\prod_{j=i+2}^m\overline{d_{j-1}})\mathcal{A}\Theta(m,i,v)+\Theta(i+1,i,v)\\
&=d_{i+1}\mathcal{A}^{-1}(\widehat{R_{g_1}})_{\mathcal{A}\overline{d_{i+1}}v}
 -\mathcal{A}^{-1}(\widehat{R_{g_1}})_{\overline{d_{i+1}}\mathcal{A}\overline{d_{i+1}}v}\\
 &+\sum_{m=i+2}^{\mathfrak{n}}\mathcal{A}^{-1}(\prod_{j=i+2}^m\overline{d_{j-1}})\mathcal{A}\Theta(m,i,v)+\Theta(i+1,i,v).
\end{align*}
%In \eqref{for:78}, letting $m=i$ and using above computation we get
%\begin{align*}
% (\widehat{R_{d_{i}}})_{A\overline{d_{i+1}}v}&=A^{-1}\overline{d_i}A(\widehat{R_{d_{i+1}}})_{A\overline{d_{i+1}}v}\\
% &+A^{-1}\overline{d_{i}}(\widehat{R_{g_1}})_{\overline{d_{i+1}}A\overline{d_{i+1}}v}
% -A^{-1}(\widehat{R_{g_1}})_{\overline{d_{i}}A\overline{d_{i+1}}v}+\Theta(i,i,v)\\
% &=-A^{-1}\overline{d_i}(\widehat{R_{g_1}})_{\overline{d_{i+1}}A\overline{d_{i+1}}v}+A^{-1}\overline{d_i}A\overline{d_{i+1}}A^{-1}(\widehat{R_{g_1}})_{A\overline{d_{i+1}}v}\\
% &+\sum_{m=i+2}^{\mathfrak{n}}A^{-1}\overline{d_i}(\prod_{j=i+2}^m\overline{d_{j-1}})A\Theta(m,i,v)+A^{-1}\overline{d_i}A\Theta(i+1,i,v)\\
% &+A^{-1}\overline{d_{i}}(\widehat{R_{g_1}})_{\overline{d_{i+1}}A\overline{d_{i+1}}v}
% -A^{-1}(\widehat{R_{g_1}})_{\overline{d_{i}}A\overline{d_{i+1}}v}+\Theta(i,i,v)\\
%  &\stackrel{(1)}{=}\overline{d_{i}}A^{-1}(\widehat{R_{g_1}})_{A\overline{d_{i+1}}v}
%  -A^{-1}(\widehat{R_{g_1}})_{\overline{d_{i}}A\overline{d_{i+1}}v}\\
% &+\sum_{m=i+2}^{\mathfrak{n}}A^{-1}\overline{d_i}(\prod_{j=i+2}^m\overline{d_{j-1}})A\Theta(m,i,v)+A^{-1}\overline{d_i}A\Theta(i+1,i,v)\\
% &+\Theta(i,i,v).
%\end{align*}
%In $(1)$ we used the relation $d_i=A^{-1}d_iAd_{i+1}$.
Hence it follows from \eqref{for:79} that
\begin{align}\label{for:81}
  &\big|(\widehat{R_{d_{i+1}}})_{\mathcal{A}\overline{d_{i+1}}v}-\overline{d_{i+1}}\mathcal{A}^{-1}(\widehat{R_{g_1}})_{\mathcal{A}\overline{d_{i+1}}v}
  +\mathcal{A}^{-1}(\widehat{R_{g_1}})_{\overline{d_{i+1}}\mathcal{A}\overline{d_{i+1}}v}\big|\notag\\
  &\leq C_a\norm{\mathcal{L}}_a\norm{v}^{-a+(\mathfrak{n}+1)\sigma}.
\end{align}
Then by using \eqref{for:88} we can estimate $J_i(v)$ by rewriting it as follows:
\begin{align*}
  J_i(v)&\stackrel{(1)}{=}\overline{d_{i+1}^{-1}}\big((\widehat{R_{d_{i+1}}})_{v}-(\widehat{R_{g_1}})_{\mathcal{A}^{-1}\overline{d_{i+1}}v}\big)\\
  &+(\mathcal{A}\overline{d_{i+1}})^{-2}\mathcal{A}\big((\widehat{R_{d_{i+1}}})_{\mathcal{A}\overline{d_{i+1}}v}-\overline{d_{i+1}}\mathcal{A}^{-1}
  (\widehat{R_{g_1}})_{\mathcal{A}\overline{d_{i+1}}v}\\
  &+\mathcal{A}^{-1}(\widehat{R_{g_1}})_{\overline{d_{i+1}}\mathcal{A}\overline{d_{i+1}}v}\big)\\
  &+\overline{d_{i+1}^{-1}}(\widehat{R_{g_1}})_{\mathcal{A}^{-1}\overline{d_{i+1}}v}+(\mathcal{A}\overline{d_{i+1}})^{-2}\mathcal{A}\overline{d_{i+1}}\mathcal{A}^{-1}
  (\widehat{R_{g_1}})_{\mathcal{A}\overline{d_{i+1}}v}\\
  &-(\mathcal{A}\overline{d_{i+1}})^{-2}\mathcal{A}\mathcal{A}^{-1}(\widehat{R_{g_1}})_{\overline{d_{i+1}}\mathcal{A}\overline{d_{i+1}}v}\\
  &+(\mathcal{A}\overline{d_{i+1}})^{-2}(\widehat{R_{g_1}})_{\overline{d_{i+1}}\mathcal{A}\overline{d_{i+1}}v}
-\overline{d_{i+1}^{-1}}(\widehat{R_{g_1}})_{\mathcal{A}^{-1}\overline{d_{i+1}}v}\notag\\
&-\overline{d_{i+1}^{-1}}\mathcal{A}^{-2}(\widehat{R_{g_1}})_{\mathcal{A}\overline{d_{i+1}}v}\\
&=\overline{d_{i+1}^{-1}}\big((\widehat{R_{d_{i+1}}})_{v}-(\widehat{R_{g_1}})_{\mathcal{A}^{-1}\overline{d_{i+1}}v}\big)\\
  &+(\mathcal{A}\overline{d_{i+1}})^{-2}\mathcal{A}\big((\widehat{R_{d_{i+1}}})_{\mathcal{A}\overline{d_{i+1}}v}-\overline{d_{i+1}}\mathcal{A}^{-1}
  (\widehat{R_{g_1}})_{\mathcal{A}\overline{d_{i+1}}v}\\
  &+\mathcal{A}^{-1}(\widehat{R_{g_1}})_{\overline{d_{i+1}}\mathcal{A}\overline{d_{i+1}}v}\big).
  \end{align*}
In $(1)$ we used \eqref{for:80}, \eqref{for:81} to substitute $(\widehat{R_{d_{i+1}}})_{v}$ and $(\widehat{R_{d_{i+1}}})_{\mathcal{A}\overline{d_{i+1}}v}$ respectively, which also implies:
 \begin{align*}
  \abs{J_i(v)}\leq C_a\norm{\mathcal{L}}_a\norm{v}^{-a+(\mathfrak{n}+2)\sigma}.
\end{align*}
Then \eqref{for:83} follows from \eqref{for:78} and above computation immediately.
\end{proof}

\begin{lemma}\label{le:4}
For any $v\in E_\mathcal{A}$, we have
\begin{align*}
 &\Big|\sum_j \mathcal{A}^{-(j+1)}(\widehat{R_{g_1}})_{\mathcal{A}^jv}\Big|\leq C_a\norm{\mathcal{L}}_a\norm{v}^{-a+\kappa},
\end{align*}
where $a>\kappa=(\mathfrak{n}+2)\sigma+(N+1)\tau_1^{-1}\log q$, $q=\max_{\delta_1,\delta_2=-1,0,1}{\norm{\mathcal{B}^{\delta_1}\mathcal{A}^{\delta_2}}}$.
\end{lemma}
\begin{proof}
Let $v_{0,+}=v_{0,-}=v$ and let
\begin{align*}
 v_{n,+}=(\mathcal{M}_\mathcal{A}\mathcal{B})^nv\quad\text{ and }\quad v_{n,-}=(\mathcal{M}_\mathcal{A}\mathcal{B}^{-1})^nv, \quad n\geq1.
\end{align*}
(see \eqref{for:105} of Section \ref{sec:1}) By Lemma \ref{le:9} and the construction at  beginning of Section \ref{sec:2}, we get
 \begin{align*}
 v_{n,+}=\mathcal{A}^{\delta^+_n}\mathcal{B}v_{n-1,+}\quad\text{ and }\quad v_{n,-}=\mathcal{A}^{\delta^-_n}\mathcal{B}^{-1}v_{n-1,-}
 \end{align*}
 where $\delta^+_n,\delta^-_n\in \{-1,0,1\}$. This shows that
\begin{align*}
 v_{n,+}=\Pi_{j=n}^1 (\mathcal{A}^{\delta^+_j}\mathcal{B})v\quad\text{ and }\quad v_{n,-}=\Pi_{j=n}^1 (\mathcal{A}^{\delta^-_j}\mathcal{B}^{-1})v,\quad n\geq1.
\end{align*}
(see \eqref{for:12} of Section \ref{sec:1}).  We obtain:
\begin{align*}
 &\sum_j \mathcal{A}^{-(j+1)}(\widehat{R_{g_1}})_{\mathcal{A}^jv_{n,\pm}}=\sum_j \mathcal{A}^{-(j+1)}(\widehat{R_{g_1}})_{\mathcal{A}^j\mathcal{A}^{\delta^{\pm}_n}\mathcal{B}^{\pm 1}v_{n-1,\pm}}\\
 &=\mathcal{A}^{\delta^{\pm}_n}\sum_j \mathcal{A}^{-(j+1)}(\widehat{R_{g_1}})_{\mathcal{A}^j\mathcal{B}^{\pm 1}v_{n-1,\pm}}.
\end{align*}
since all the sums involved converge absolutely by \eqref{for:26} of Corollary \ref{cor:2}.

Let $\mathcal{E}^{\pm}_n=\prod_{i=1}^{n}(\mathcal{B}^{\mp 1}\mathcal{A}^{-\delta^{\pm }_i})$, then formally we obtain:
\begin{align*}
 &\sum_j \mathcal{A}^{-(j+1)}(\widehat{R_{g_1}})_{A^jv}- \lim_{n\rightarrow\infty}\mathcal{E}^{\pm}_n\mathcal{B}^{\mp 1}\sum_j \mathcal{A}^{-(j+1)}(\widehat{R_{g_1}})_{\mathcal{A}^j\mathcal{B}^{\pm 1}v_{n,\pm}}\\
 &=\mathcal{D}_{v,\pm}\stackrel{\text{def}}{=}\sum_j \mathcal{A}^{-(j+1)}(\widehat{R_{g_1}})_{A^jv}-\sum_j \mathcal{B}^{\mp 1}\mathcal{A}^{-(j+1)}(\widehat{R_{g_1}})_{\mathcal{A}^j\mathcal{B}^{\pm 1}v}\\
 &+\sum_{n=1}^\infty \mathcal{E}^{\pm}_n\big(\sum_j \mathcal{A}^{-(j+1)}(\widehat{R_{g_1}})_{A^jv_{n,\pm}}-\mathcal{B}^{\mp 1}\sum_j \mathcal{A}^{-(j+1)}(\widehat{R_{g_1}})_{\mathcal{A}^j\mathcal{B}^{\pm1}v_{n,\pm}}\big)
\end{align*}
%and
%\begin{align*}
% &\sum_j \mathcal{A}^{-(j+1)}(\widehat{R_{g_1}})_{A^jv}- \lim_{n\rightarrow\infty}\mathcal{E}'_n\mathcal{B}\sum_j \mathcal{A}^{-(j+1)}(\widehat{R_{g_1}})_{\mathcal{A}^j\mathcal{B}^{-1}v_{n,-}}\\
% &=\mathcal{D}_{v,-}\stackrel{\text{def}}{=}\sum_j \mathcal{A}^{-(j+1)}(\widehat{R_{g_1}})_{A^jv}-\sum_j \mathcal{B}\mathcal{A}^{-(j+1)}(\widehat{R_{g_1}})_{\mathcal{A}^j\mathcal{B}^{-1}v}\\
% &+\sum_{n=1}^\infty \mathcal{E}'_n\big(\sum_j \mathcal{A}^{-(j+1)}(\widehat{R_{g_1}})_{A^jv_{n,-}}-\mathcal{B}\sum_j \mathcal{A}^{-(j+1)}(\widehat{R_{g_1}})_{\mathcal{A}^j\mathcal{B}^{-1}v_{n,-}}\big),
%\end{align*}
Next, we will justify the convergence of $\mathcal{D}_{v,+}$ and $\mathcal{D}_{v,-}$ and compute the two limits. To do so, we need to estimate the growth rate of
$v_{n,\pm}$. Let $l^{\pm}_n=\sum_{i=1}^n\delta^{\pm}_i$. Similar to \eqref{for:51} we can write
\begin{align}\label{for:95}
 \Pi_{j=n}^1 (\mathcal{A}^{\delta^{\pm}_j}\mathcal{B}^{\pm1})=e_{n}^{\pm}\mathcal{A}^{l^{\pm}_{n}}\mathcal{B}^{\pm n},\qquad \forall n>0.
\end{align}
where $e_{n}^{\pm}\in \alpha([\mathcal{H},\mathcal{H}])$.

By \eqref{for:115} of Proposition \ref{po:1}
\begin{align}\label{for:116}
 \norm{e_{n}^{\pm}}\geq Cp(n)^{-1}
\end{align}
for a polynomial $p$.

Using Lemma \ref{le:1} and above inequality, for any $u\in\ZZ^N\backslash 0$ we have
\begin{align}\label{for:100}
 \norm{v_{n,\pm}}&=\big\|e_{n}^{\pm}(\mathcal{A}^{l^{\pm}_n}\mathcal{B}^{\pm n})v\big\|\geq Ce^{\tau n}p(n)^{-1}\norm{v}^{-N}\notag\\
 &\geq Ce^{\tau n/2}\norm{v}^{-N}.
 \end{align}
Then apply Proposition \ref{po:2} and \eqref{for:100}, we have
\begin{align}\label{for:98}
  \norm{\mathcal{D}_{v,\pm}}&\leq C_a\sum_{n=0}^{+\infty} q^n\norm{\mathcal{L}}_a\norm{v_{n,\pm}}^{-a+(\mathfrak{n}+1)\sigma}\\
  &\leq C_a\sum_{n=0}^{+\infty} q^n\norm{\mathcal{L}}_a(C\exp\{\tau n/2\}\norm{v}^{-N})^{-a+(\mathfrak{n}+1)\sigma}\notag\\
  &\leq C_a\sum_{n=0}^{+\infty} q^n\norm{\mathcal{L}}_a\exp\{-\tau n(a-(\mathfrak{n}+1)\sigma)/2\}\norm{v}^{N(a-(\mathfrak{n}+1)\sigma)}\notag\\
  &<\infty\notag
\end{align}
if $a>\frac{2\log q}{\tau }+(\mathfrak{n}+1)\sigma$, where $q=\max_{\delta_1,\delta_2=-1,0,1}{\norm{\mathcal{B}^{\delta_1}\mathcal{A}^{\delta_2}}}$.

To estimate the limit we use \eqref{for:34} of Corollary \ref{cor:2} and \eqref{for:100}. The two conditions are satisfied (see the beginning of the proof of Proposition \ref{po:2}). Then we have
\begin{align*}
 &\big\|\mathcal{E}^{\pm}_n\mathcal{B}^{\mp 1}\sum_j \mathcal{A}^{-(j+1)}(\widehat{R_{g_1}})_{\mathcal{A}^j\mathcal{B}^{\pm1}v_{n,\pm}}\big\|\\
 &\leq Cq^n \big\| \sum_j\mathcal{A}^{-(j+1)}(\widehat{R_{g_1}})_{\mathcal{A}^j\mathcal{B}^{\pm 1}v_{n,\pm}}\big\|\\
 &\leq C_aq^n\norm{\widehat{R_{g_1}}}_a \norm{\mathcal{B}v_{n,\pm}}^{-a+\kappa_{\mathcal{A},\mathcal{A}}+1}\\
 &\leq C_aq^n\norm{\widehat{R_{g_1}}}_a \exp\{-\tau\abs{n}(a-\kappa_{\mathcal{A},\mathcal{A}}-1)/2\}\norm{v}^{N(a-\kappa_{\mathcal{A},\mathcal{A}}-1)}
\end{align*}
for any $a>\frac{2\log q}{\tau }+\kappa_{\mathcal{A},\mathcal{A}}+1$.

Then it follows that the limit
\begin{align*}
 \lim_{n\rightarrow\infty}\mathcal{E}^{\pm}_n\mathcal{B}^{\mp 1}\sum_j \mathcal{A}^{-(j+1)}(\widehat{R_{g_1}})_{\mathcal{A}^j\mathcal{B}^{\pm1}v_{n,\pm}}=0.
\end{align*}
Hence the obstruction has two expressions:
\begin{align}\label{for:99}
 \sum_j \mathcal{A}^{-(j+1)}(\widehat{R_{g_1}})_{A^jv}=\mathcal{D}_{v,+}=\mathcal{D}_{v,-}.
\end{align}
We estimate the obstruction using both of its forms in order to obtain needed estimates.

 Next, we will estimate the lower bound of the increasing speed of $v_{n,\pm}$  in terms of $\norm{v}$ instead of $\norm{v}^{-N}$. \eqref{for:95} allows us to estimate growth rate of  $(\mathcal{A}^{l^{\pm}_n}\mathcal{B}^{\pm n})v$ instead.

For any $v=\mathcal{M}(v)$, in case $\mathcal{A}v\hookrightarrow 2(\mathcal{A})$, let $v_1$ be the projection of $v$ to the $0$-Lyapunov space $J'$ for $\mathcal{A}$. Then
\begin{align*}
  \norm{v_1}\geq C\norm{v}.
\end{align*}
Let the Lyapunov exponent of $\mathcal{B}$ on this Lyapunov space be $\nu$. For all $k=(k_1,k_2)\in\ZZ^2$ the Lyapunov exponent of $\mathcal{A}^{k_1}\mathcal{B}^{k_2}$ on $J'$ is $k_2\nu$ (see Proposition \ref{po:1}).
Then if $\nu\geq 0$, on the half-space $K_v^+=\{{k_1,k_2}\in\ZZ^2:k_2\geq 0\}$ we obtain
\begin{align*}
 \norm{\mathcal{A}^{k_1}\mathcal{B}^{k_2}v}\geq C(\abs{k_1}+1)^{-N}(\abs{k_2}+1)^{-N}\norm{v_1}\geq C(\abs{k_1k_2}+1)^{-N}\norm{v}
\end{align*}
and
\begin{align*}
 \norm{\mathcal{A}^{k_1}\mathcal{B}^{k_2}v}\geq C(\abs{k_1}+1)^{-N}e^{k_2\nu/2}\norm{v_1}\geq C(\abs{k_1k_2}+1)^{-N}\norm{v}
\end{align*}
on $K^-_v=\{{k_1,k_2}\in\ZZ^2:k_2\leq 0\}$ if $\nu< 0$.

In case $\mathcal{A}v\hookrightarrow 1(\mathcal{A})$, let $v_1$ and $v_2$ be the largest projections of $v$ to some Lyapunov space $J_1$ and $J_2$ with positive Lyapunov exponent $\lambda_1$
and negative Lyapunov exponent $\lambda_2$ respectively. Let $\nu_1$ and $\nu_2$ be corresponding Lyapunov exponents of $\mathcal{B}$ on the two Lyapunov spaces. Then
\begin{align*}
  \norm{v_1}\geq C\norm{v},\qquad \norm{v_2}\geq C\norm{v}.
\end{align*}
For all $k=(k_1,k_2)\in\ZZ^2$ the Lyapunov exponent of $\mathcal{A}^{k_1}\mathcal{B}^{k_2}$  on $J_1$ is $\chi(k_1,k_2)^+=k_1\lambda_1+k_2\nu_1$ and is $\chi(k_1,k_2)^-=k_2\lambda_2+k_2\nu_2$ on $J_2$. We want to obtain:
\begin{align}\label{for:28}
 \{(k_1,k_2):\chi(k_1,k_2)^+\geq0\}\bigcup \{(k_1,k_2):\chi(k_1,k_2)^-\geq0\}
\end{align}
covers either $K_v^+$ or $K_v^-$, which boils down to require $k_2(\frac{\nu_1}{\lambda_1}-\frac{\nu_2}{\lambda_2})\geq0$. Namely, for
any $(k_1,\,k_2)\in\ZZ^2$, $(k_1,\,k_2)$ belongs to the union in \eqref{for:28} if $k_2(\frac{\nu_1}{\lambda_1}-\frac{\nu_2}{\lambda_2})\geq0$
and this is true for $k_2\geq 0$ or for $k_2\leq 0$ depending on the sign of $\frac{\nu_1}{\lambda_1}-\frac{\nu_2}{\lambda_2}$. Therefore we obtain
\begin{align}\label{for:29}
 \norm{\mathcal{A}^{k_1}\mathcal{B}^{k_2}v}\geq C(\abs{k_1}+1)^{-N}(\abs{k_2}+1)^{-N}\norm{v}.
\end{align}
in $K_v^+$ or in $K_v^-$.

The above analysis shows that for any $v=\mathcal{M}(v)$, we can choose $\text{sgn}(v)\in\{-1,1\}$ such that \eqref{for:29} holds on $K_v^{\text{sgn}(v)}$. Using \eqref{for:95}, \eqref{for:116} and \eqref{for:29} we obtain:
\begin{align}\label{for:101}
  &\norm{v_{n,\text{sgn}(v)}}=\big\|e_{n}^{\text{sgn}(v)}\mathcal{A}^{l^{\text{sgn}(v)}_{n}}\mathcal{B}^{\text{sgn}(v) n}v\big\|\notag\\
 &\geq C(n^2+1)^{-N}p(n)^{-1}\norm{v}
 \end{align}
on $K_v^{\text{sgn}(v)}$. From \eqref{for:98} we get
\begin{align}\label{for:102}
  \norm{\mathcal{D}_{v,\text{sgn}(v)}}&\leq C_a\sum_{n=0}^{+\infty} q^n\norm{\mathcal{L}}_a\norm{v_{n,\text{sgn}(v)}}^{-a+(\mathfrak{n}+1)\sigma}
 \end{align}
if $a>(\mathfrak{n}+1)\sigma$.

 From \eqref{for:99}, \eqref{for:100}, \eqref{for:101} and \eqref{for:102} we see that the conclusion follows immediately if we can prove the following claim:

 \noindent Suppose $u_n$, $n\geq 0$ is a sequence of vectors in $\ZZ^N\backslash 0$ such that
 \begin{align*}
  (*)\norm{u_{n}}\geq Ce^{n\tau/2}\norm{u_0}^{-N}\quad\text{ and }\quad(**)\norm{u_{n}}\geq C(n+1)^{-2N}p(n)^{-1}\norm{u_0}.
 \end{align*}
 Then for any $a>\max\{(\mathfrak{n}+1)\sigma, (\mathfrak{n}+1)\sigma+\frac{2}{\tau}\log q\}$
 \begin{align*}
 \sum_{n\geq0} q^n\norm{u_{n}}^{-a+(\mathfrak{n}+1)\sigma}\leq C_{a,\delta}\norm{u_0}^{-a+(\mathfrak{n}+1)\sigma+\delta}
 \end{align*}
 for any $\delta>0$.

\emph{Proof of the claim}. Let $\tau_1=\frac{\tau}{2}$ and  $n_0=[\frac{(N+1)}{\tau_1}\log\norm{u_0}]+1$. We have
\begin{align*}
  &\sum_{n\geq0} q^n\norm{u_{n}}^{-a+(\mathfrak{n}+1)\sigma}\\
  &=\sum_{n\geq n_0+1} q^n\norm{u_{n}}^{-a+(\mathfrak{n}+1)\sigma}+\sum_{n\leq n_0}q^n\norm{u_{n}}^{-a+(\mathfrak{n}+1)\sigma}\\
   & \stackrel{(1)}{\leq}C_a\sum_{n\geq n_0+1}q^n (e^{\tau_1 n}\norm{u_0}^{-N})^{-a+(\mathfrak{n}+1)\sigma}\\
  &+n_0q^{n_0}C_{a} (n_0+1)^{2N}p(n_0)\norm{u_0}^{-a+(\mathfrak{n}+1)\sigma}\\
  & \leq C_a\sum_{n\geq n_0+1}q^n \big(e^{\tau_1(n-n_0)}(e^{\tau_1n_0}\norm{u_0}^{-N})\big)^{-a+(\mathfrak{n}+1)\sigma}\\
  &+C_{a,\delta} \norm{u_0}^{-a+(\mathfrak{n}+1)\sigma+\delta+(N+1)\tau_1^{-1}\log q}\\
  &\stackrel{(2)}{\leq} C_a\sum_{n\geq n_0+1}q^n \big(e^{\tau_1(n-n_0)}\norm{u_0}\big)^{-a+(\mathfrak{n}+1)\sigma}\\
  &+C_{a,\delta}\norm{u_0}^{-a+(\mathfrak{n}+1)\sigma+\delta+(N+1)\tau_1^{-1}\log q}\\
  &\stackrel{(3)}{\leq} C_{a,\delta}\norm{u_0}^{-a+(\mathfrak{n}+1)\sigma+\delta+(N+1)\tau_1^{-1}\log q}
\end{align*}
for any $\delta>0$ providing $a>\max\{(\mathfrak{n}+1)\sigma, (\mathfrak{n}+1)\sigma+(N+1)\tau_1^{-1}\log q\}$.

Here in $(1)$ to estimate $\sum_{n\geq n_0+1}$ we use the exponential increasing
estimate $(*)$ and for $\sum_{n\leq n_0}$ the we use the lower bound in $(**)$. In $(2)$ we rewrite
\begin{align*}
  q^{n} (e^{\tau_1(n-n_0)}\norm{u_0})^{-a+(\mathfrak{n}+1)\sigma}=q^{n_0} \big(e^{(n-n_0)(\tau_1-\frac{\log q}{a-(\mathfrak{n}+1)\sigma})}\norm{u_0}\big)^{-a+(\mathfrak{n}+1)\sigma}
\end{align*}
and note that polynomials increase slower than any exponential functions, then $(3)$ follows immediately.
\end{proof}

\begin{proposition}\label{po:3}
Fix $\sigma=N+3+\kappa$. There exists $\delta>0$ such that for any $C^\infty$ maps $\theta$, $\psi$, $\omega$ on $\TT^N$ that are $C^\sigma$ small enough,
it is possible to split $R_{g_1}$, $R_{d_i}$, $0\leq d\leq \mathfrak{n}$ as
\begin{gather*}
R_{g_1}=\Delta_\mathcal{A}\Omega+\mathcal{R}R_{g_1},\qquad R_{d_i}=\Delta_{\overline{d_i}}\Omega+\mathcal{R}R_{d_i}
 \end{gather*}
for a $C^\infty$ map $\Omega$, so that
\begin{align*}
 \norm{\Omega}_{C^r}&\leq C_{r}\norm{R_{g_1},R_{d_0},\cdots,R_{d_\mathfrak{n}}}_{C^{r+\sigma}}
\end{align*}
and
\begin{align*}
 \norm{\mathcal{R}R_{g_1},\mathcal{R}R_{d_0},\cdots,\mathcal{R}R_{d_\mathfrak{n}}}_{C^r}&\leq C_{r}\norm{\mathcal{L}(g_1,d_0),\cdots,\mathcal{L}(g_1,d_{\mathfrak{n}})}_{C^{r+\sigma+()\sigma_1}}
 \end{align*}
for any $r\geq 0$.
\end{proposition}
\begin{proof}
$\mathcal{R}\theta$ and $\Omega$ are constructed in \eqref{for:3} and \eqref{for:30} respectively. The $C^r$ estimate for $\Omega$ follow immediately from \eqref{for:16} and \eqref{for:110} of Section \ref{sec:1}; and $C^r$ estimate for $\mathcal{R}R_{g_1}$ are from Lemma \ref{le:4} and \eqref{for:110} of Section \ref{sec:1} we get
\begin{align*}
 \norm{\mathcal{R}R_{g_1}}_{C^r}&\leq C_{r}\norm{\mathcal{L}(g_1,d_0),\cdots,\mathcal{L}(g_1,d_{\mathfrak{n}})}_{C^{r+\sigma}}
 \end{align*}
for any $r\geq 0$.

Let
\begin{align*}
 \mathcal{R}R_{d_i}=R_{d_i}-\Delta_{\overline{d_i}}\Omega,\qquad 0\leq i\leq\mathfrak{n}.
\end{align*}
It is easy to check that if we substitute $R_{d_i}$ by $\mathcal{R}R_{d_i}$ \eqref{for:60}  still holds.
In \eqref{for:60} letting $i=\mathfrak{n}$, we get
\begin{align*}
  &\mathcal{A}\mathcal{R}R_{d_{\mathfrak{n}}}-\mathcal{R}R_{d_\mathfrak{n}}\circ \mathcal{A}=\overline{d_\mathfrak{n}}\mathcal{R}R_{g_1}-\mathcal{R}R_{g_1}\circ \overline{d_\mathfrak{n}}+\mathcal{L}(g_1,d_\mathfrak{n}).
  \end{align*}
Then the following estimate holds
\begin{align*}
 \norm{\mathcal{R}R_{d_\mathfrak{n}}}_{C^r}&\leq C_{r}\norm{\overline{d_\mathfrak{n}}\mathcal{R}R_{g_1}-\mathcal{R}R_{g_1}\circ \overline{d_\mathfrak{n}}+\mathcal{L}(g_1,d_\mathfrak{n})}_{C^{r+\sigma}}\\
 &\leq C_{r}\norm{\mathcal{L}(g_1,d_0),\cdots,\mathcal{L}(g_1,d_{\mathfrak{n}})}_{C^{r+2\sigma}}
 \end{align*}
for any $r\geq 0$.

Now we proceed by induction. Fix $i$ between $\mathfrak{n}$ and $0$ and assume that for
all $j\geq i$
\begin{align*}
 \norm{\mathcal{R}R_{d_j}}_{C^r}&\leq C_{r}\norm{\mathcal{L}(g_1,d_0),\cdots,\mathcal{L}(g_1,d_{\mathfrak{n}})}_{C^{r+(\mathfrak{n}-j+2)\sigma}}
 \end{align*}
for any $r\geq 0$.

By using \eqref{for:60} for $i-1$ we get
\begin{align*}
 &\norm{\mathcal{R}R_{d_{i-1}}}_{C^r}\\
 &\leq C_{r}\norm{\overline{d_{i-1}}\mathcal{R}R_{g_1}\circ \overline{d_{i}}-\mathcal{R}R_{g_1}\circ \overline{d_{i-1}}+\overline{d_{i-1}}A\mathcal{R}R_{d_{i}}+\mathcal{L}(g_1,d_{i-1})}_{C^{r+\sigma}}\\
 &\leq C_{r}\norm{\mathcal{L}(g_1,d_0),\cdots,\mathcal{L}(g_1,d_{\mathfrak{n}})}_{C^{r+\sigma+(\mathfrak{n}-i+2)\sigma}}\\
 &=C_{r}\norm{\mathcal{L}(g_1,d_0),\cdots,\mathcal{L}(g_1,d_{\mathfrak{n}})}_{C^{r+(\mathfrak{n}-(i-1)+2)\sigma}}
 \end{align*}
for any $r\geq 0$.

Then we get the estimate for the case of $i-1$. Hence we obtain the estimates for all $\mathcal{R}R_{d_i}$, $0\leq i\leq\mathfrak{n}$.

\end{proof}
\section{Construction of the projection for action $\alpha$ when $\mathcal{H}$ is not nilpotent}\label{sec:4}

\begin{proposition}\label{po:4}
Fix $\sigma=N+3+\kappa$. There exists $\delta>0$ such that for any $C^\infty$ maps $\theta$, $\psi$, $\omega$ on $\TT^N$ that are $C^\sigma$ small enough,
it is possible to split $R_{A_i}$, $1\leq i\leq 4$ as $R_{A_i}=\Delta_{A_i}\Omega+\mathcal{R}R_{A_i}$,  $1\leq i\leq 4$ for a $C^\infty$ map $\Omega$, so that
\begin{align*}
 \norm{\Omega}_{C^r}&\leq C_{r}\max_{1\leq i\leq 4}\norm{R_{A_i}}_{C^{r+\sigma}},\quad \forall\,r\geq0
\end{align*}
and
\begin{align*}
 \max_{1\leq i\leq4}\norm{\mathcal{R}R_{A_i}}_{C^r}&\leq C_{r}\max_{2\leq i\leq4}\norm{\mathcal{L}(A_1,A_i)}_{C^{r+\sigma+()\sigma_1}},\quad \forall\,r\geq0.
 \end{align*}
 \end{proposition}
\begin{proof}
Similar to \eqref{for:3} we define $\mathcal{R}R_{A_1}=\sum_{v\in\ZZ^N}(\widehat{\mathcal{R}R_{A_1}})_ve_v$ where
\begin{align*}
(\widehat{\mathcal{R}R_{A_1}})_v&=\left\{\begin{aligned} &\sum_{i\in\ZZ}A^{-i}(\widehat{\mathcal{R}R_{A_1}})_{A^iv},\qquad &v=\mathcal{M}_{A_1}(v),\\
&0, \qquad &\text{otherwise}
\end{aligned}
 \right.
\end{align*}
for $v\neq 0$ and $(\widehat{\mathcal{R}R_{A_1}})_0=0$.

Since $R_{A_1}-\mathcal{R}R_{A_1}$ satisfies the solvable condition in Lemma \ref{le:5}, by using Lemma \ref{le:5} there is a $C^\infty$ function $\Omega$ such that
\begin{align*}
 \Delta_A\Omega=R_{A_1}-\mathcal{R}R_{A_1}
\end{align*}
with estimates
\begin{align*}
 \norm{\Omega}_{a}\leq C_a\norm{R_{A_1}-\mathcal{R}R_{A_1}}_{a+\sigma_1}\leq C_a\norm{R_{A_1}}_{a+2\sigma_1}, \quad\forall a\geq 0.
\end{align*}
where $\sigma_1>\kappa_{A,A}$.

Since
\begin{align}\label{for:36}
  A_1R_{A_i}-R_{A_i}\circ A_1=A_iR_{A_1}-R_{A_1}\circ A_i+\mathcal{L}_i
  \end{align}
 for $2\leq i\leq n$, the obstructions for each side with respect to $A_1$ vanish; therefore for any $v=\mathcal{M}_{A_1}(v)$ we have
 \begin{align*}
  \sum_{j\in\ZZ}A_1^{-(j+1)}A_i(\widehat{\mathcal{R}R_{A_1}})_{A_1^jv}-\sum_{j\in\ZZ}A_1^{-(j+1)}(\widehat{\mathcal{R}R_{A_1}})_{A_1^jA_iv}
  =-\sum_{j\in\ZZ}A_1^{-(j+1)}(\widehat{\mathcal{L}_i})_{A_1^jv}.
 \end{align*}
The absolutely convergence of all sums evolved is guaranteed by
\eqref{for:26} of Corollary \ref{cor:2}.

Recall nations above Theorem \ref{th:2}. Since $\bigcap_{i=2}^np_1(A_i)=\{0\}$, $\bigoplus_{i=2}^np_1(A_i)^{\bot}=\RR^N$. Let $P_i$, $2\leq i\leq n$ be the projection to $p_1(A_i)^{\bot}$. For any $v\in E_{A_1}$, there exists  $2\leq i_0\leq n$ such that $\norm{P_{i_0}(v)}\geq C\norm{v}$.

Iterating this equation with respect to $A_{i_0}$ we obtain
 \begin{align*}
  &\sum_{j\in\ZZ}A_1^{-j}(\widehat{\mathcal{R}R_{A_1}})_{A_1^jv}
  -\sum_{l\rightarrow\infty}\sum_{j\in\ZZ}A_1^{-j}A_{i_0}^{-l}(\widehat{\mathcal{R}R_{A_1}})_{A_1^jA_{i_0}^{l}v}\\
  &=-\sum_{l=0}^\infty\sum_{j\in\ZZ}A_{i_0}^{-(l+1)}A_1^{-(j+1)}(\widehat{\mathcal{L}_i})_{A_1^jA_{i_0}^lv}.
 \end{align*}
 Since $A_{i_0}v\neq v$, it follows from \eqref{for:32} of Corollary \ref{cor:4} that all the involving sums are convergent absolutely, which also implies that the limit above is $0$. By iterating backwards and applying the
same reasoning, in the notation of Corollary \ref{cor:4} we obtain
\begin{align}
  \sum_{j\in\ZZ}A_1^{-j}(\widehat{\mathcal{R}R_{A_1}})_{A_1^ju}&=-S_{K^+}(\mathcal{L}_i,u)(A_1,A_{i_0})=S_{K^-}(\mathcal{L}_i,u)(A_1,A_{i_0}),
 \end{align}
where $K^+=\{{k_1,k_2}\in\ZZ^2:k_2\geq 0\}$ and $K^-=\{{k_1,k_2}\in\ZZ^2:k_2<0\}$.

 Then according to \eqref{for:25} of Corollary \ref{cor:4}, the needed estimate for $\mathcal{R}R_{A_1}$ with respect to $\mathcal{L}$ follows if in at least on one of the half-spaces $K^+$ and $K^-$ the
dual action satisfies some polynomial lower bound for every $u\in E_{A_1}$ (see \eqref{for:105} of Section \ref{sec:1}).

For $A_1$ we have the following eigenspace decomposition:
\begin{align*}
  \RR^N=\bigoplus_{i\in I} J_{\mu_i}
\end{align*}
where $J_{\mu_i}$ is the eigenspace  of $A_1$ with eigenvalue $\mu_i$. We denote by $q_i$ the projection to each $J_{\mu_i}$.

Since $A_1$ and $A_{i_0}$ commute, $A_{i_0}$ preserves each $J_{\mu_i}$.  We can choose a basis of $\RR^N$ such that $A_{i_0}$ has its Jordan normal form on each $J_{\mu_i}$. Then for any $u\in\RR^N$ we have a decomposition $u=u_1+u_2$ determined by this basis, where $A_{i_0}u_1=u_1$, $(A_{i_0}-I)u_2\neq0$ and
\begin{align*}
\norm{A_{i_0}^kq_i(u_2)}\geq C\norm{q_i(u_2)}, \qquad \forall k\in\ZZ,\,\forall i\in I
\end{align*}
Note that for any $u\in p_1(A_{i_0})^{\bot}$, we have
\begin{align*}
  \norm{u_2}\geq C\norm{u}.
\end{align*}
Especially, we have
\begin{align*}
 \norm{(P_{i_0}(v))_2}\geq C\norm{P_{i_0}(v)}\geq C\norm{v}.
\end{align*}
In case
$(P_{i_0}(v))_2\hookrightarrow 1,2(A_1)$, suppose $(P_{i_0}(v))_2$ has the largest projections to some Lyapunov space $J_{\mu_j}$ of $A_1$ with non-negative Lyapunov exponent. Then
\begin{align*}
 \norm{q_j(P_{i_0}(v))_2}\geq C\norm{(P_{i_0}(v))_2}\geq C\norm{v},
\end{align*}
and thus for any $k_1\geq0$
\begin{align*}
  \norm{A_1^{k_1}A_{i_0}^{k_2}v}&\geq C\norm{A_1^{k_1}A_3^{k_2}v_2}=C\norm{A_1^{k_1}A_3^{k_2}(P_{i_0}(v))_2}\geq C\norm{A_1^{k_1}A_3^{k_2}q_j(P_{i_0}(v))_2}\\
  &\geq C(k_1+1)^{-N}\norm{A_3^{k_2}q_j(P_{i_0}(v))_2}\geq C(k_1+1)^{-N}\norm{q_j(P_{i_0}(v))_2}\\
  &\geq C(k_1+1)^{-N}\norm{v}.
\end{align*}
Since $v=\mathcal{M}_{A_1}(v)$, letting $v'$ be the largest projection of $v$ to some Lyapunov space $J_{\mu_i}$ of $A_1$ with negative Lyapunov exponent $\mu_i$, we have
\begin{align*}
 \norm{v'}\geq C\norm{v}.
\end{align*}
Let $J$ be the Jorden block of $A_{i_0}$ on $J_{\mu_i}$ on which $v'$ has the largest projection $v''$. Then
\begin{align*}
  \norm{v''}\geq C\norm{v}.
\end{align*}
Since $(A_{i_0}-I)^2=0$, then $J$ is either a $1\times1$ or a $2\times2$ matrix. In case $J$ is a $1\times 1$ matrix, then for any $k_1<0$, $k_2\in \ZZ$ we have
\begin{align*}
  \norm{A_1^{k_1}A_{i_0}^{k_2}v}&\geq C\norm{A_1^{k_1}A_{i_0}^{k_2}v'}\geq Ce^{k_1\mu_i/2}\norm{A_{i_0}^{k_2}v'}\geq Ce^{k_1\mu_i/2}\norm{A_{i_0}^{k_2}v''}\\
  &=Ce^{k_1\mu_i/2}\norm{v''}\geq Ce^{k_1\mu_i/2}\norm{v}.
\end{align*}
In case $J$ is a $2\times 2$ matrix, denote the coordinates of $v''$ by $(a_1,a_2)$. If $a_1a_2\geq 0$ then for any $k_1<0$, $k_2\geq 0$ we have
\begin{align*}
  \norm{A_1^{k_1}A_{i_0}^{k_2}v}&\geq Ce^{k_1\mu_i/2}\norm{A_{i_0}^{k_2}v''}=Ce^{k_1\mu_i/2}\norm{(a_1+k_2a_2,a_2)}\geq Ce^{k_1\mu_i/2}\norm{v''}\\
  &\geq Ce^{k_1\mu_i/2}\norm{v}.
\end{align*}
If $a_1a_2<0$ then same estimate holds for any $k_1<0$, $k_2>0$. Hence
\begin{align}\label{for:33}
  \norm{A_1^{k_1}A_{i_0}^{k_2}v}\geq C(\abs{k_1}+1)^{-N}\norm{v}
\end{align}
holds either on $K^+$ or on $K^-$.

In case
$(P_{i_0}(v))_2\hookrightarrow 3(A_1)$, applying the same reasoning, we see that \eqref{for:33} still holds either on $K^+$ or on $K^-$. Now choose the half-space $K^+$ or $K^-$ in which the estimates \eqref{for:33} holds, i.e., choose
one of the sums $S_{K^+}(\varphi,v)(A_1,A_{i_0})$ or $S_{K^-}(\varphi,v)(A_1,A_{i_0})$. Then the assumption of \eqref{for:25} of Corollary \ref{cor:4} is satisfied for one of the sums above  and therefore the estimate
for $\mathcal{R}R_{A_1}$ follows:
\begin{align*}
 \norm{\mathcal{R}R_{A_1}}_{a}&\leq C_{a,\delta}\max_{2\leq i\leq n}\norm{\mathcal{L}_i}_{a+\kappa+\delta},\quad \forall\,a\geq0,\,\delta>0
 \end{align*}
where $\kappa=(n_1+1)\abs{\log_\rho y}+4(n_1+1)(1+N)$; here $n_1=2N^2+3N$, $y=\max\{\norm{A_1},\norm{A_1^{-1}}\}$.

 Let $\mathcal{R}R_{A_i}=R_{A_i}-\Delta_{A_i}\Omega$, $2\leq i\leq n$. It is easy to check that if substituting $R_{A_i}$ by $\mathcal{R}R_{A_i}$ equation \eqref{for:36}
 are also satisfied. Then it follows from Lemma \ref{le:5} that
 \begin{align*}
  \norm{\mathcal{R}R_{A_i}}_a&\leq C_a\norm{A_iR_{A_1}-R_{A_1}\circ A_i+\mathcal{L}_i}_{a+\kappa_{A_1,A_1}}\\
  &\leq C_{a,\delta}\norm{\mathcal{L}_i}_{a+\kappa_{A_1,A_1}+\kappa+\delta}.
 \end{align*}

\end{proof}
\section{error}
The following lemma shows that $\mathcal{L}$ cannot be large if $\alpha+R$ is a nilpotent group action. It is in fact quadratically small with respect to $R$.
\begin{lemma}
If $\widetilde{\alpha}=\alpha+R$ is a $C^\infty$ nipotent group action on $\TT^N$ then for any $r\geq 0$
\begin{align*}
 \norm{\mathcal{L}(g_1,d_i)}_{C^r}&\leq C_r\norm{R_{g_1},R_{d_0},\cdots,R_{d_\mathfrak{n}}}_{C^{r}}\norm{R_{g_1},R_{d_0},\cdots,R_{d_\mathfrak{n}}}_{C^{r+1}}\\
 &+C_r\norm{R_{g_1},R_{d_0},\cdots,R_{d_\mathfrak{n}}}_{C^{r+1}}\norm{R_{g_1},R_{d_0},\cdots,R_{d_\mathfrak{n}}}_{C^{r+2}}^2.
\end{align*}
\end{lemma}
\begin{proof}
Since $\mathcal{A}\overline{d_{\mathfrak{n}}}=\overline{d_{\mathfrak{n}}}\mathcal{A}$, it follows from Lemma 4.7 in \cite{Damjanovic4} that
\begin{align*}
 \norm{\mathcal{L}(g_1,d_i)}_{C^r}\leq C_r\norm{R_{g_1},R_{d_\mathfrak{n}}}_{C^r}\norm{R_{g_1},R_{d_\mathfrak{n}}}_{C^{r+1}}.
\end{align*}
Note that $\mathcal{A}\overline{d_i}=\overline{d_i}\mathcal{A}\overline{d_{i+1}}$, then
\begin{align*}
  \widetilde{\alpha}_\mathcal{A}\circ \widetilde{\alpha}_{\overline{d_i}}&=\widetilde{\alpha}_{\overline{d_i}}\circ \widetilde{\alpha}_{\mathcal{A}}\circ \widetilde{\alpha}_{\overline{d_{i+1}}}\\
  (\mathcal{A}+R_{g_1})\circ (\overline{d_i}+R_{d_i})&=(\overline{d_i}+R_{d_i})\circ (\mathcal{A}+R_{g_1})\circ (\overline{d_{i+1}}+R_{d_{i+1}}).
  \end{align*}
Then
\begin{align*}
 \mathcal{A}R_{d_i}+R_{g_1}\circ (\overline{d_i}+R_{d_i})&=R_{d_i}\circ \big(\mathcal{A}\overline{d_{i+1}}+\mathcal{A}R_{d_{i+1}}+R_{g_1}\circ (\overline{d_{i+1}}+R_{d_{i+1}})\big)\\
 &+\overline{d_i}\mathcal{A}R_{d_{i+1}}+\overline{d_i}R_{g_1}\circ (\overline{d_{i+1}}+R_{d_{i+1}}).
\end{align*}
Therefore,
\begin{align*}
  &\mathcal{L}(g_1,d_i)\\
  &=\mathcal{A}R_{d_i}-R_{d_i}\circ \mathcal{A}\overline{d_{i+1}}-\overline{d_i}R_{g_1}\circ \overline{d_{i+1}}
  +R_{g_1}\circ \overline{d_i}-\overline{d_i}AR_{d_{i+1}}\\
  &=R_{d_i}\circ \big(\mathcal{A}\overline{d_{i+1}}+\mathcal{A}R_{d_{i+1}}+R_{g_1}\circ (\overline{d_{i+1}}+R_{d_{i+1}})\big)-R_{d_i}\circ \mathcal{A}\overline{d_{i+1}}\\
  &+\overline{d_i}R_{g_1}\circ (\overline{d_{i+1}}+R_{d_{i+1}})-\overline{d_i}R_{g_1}\circ \overline{d_{i+1}}\\
  &+R_{g_1}\circ \overline{d_i}-R_{g_1}\circ (\overline{d_i}+R_{d_i}).
\end{align*}
The estimate for $C^r$ norms follows similarly (see for example [\cite{LAZUTKIN}, Appendix II]):
\begin{align*}
 &\norm{\mathcal{L}(g_1,d_i)}_{C^r}\\
 &\leq\big\|R_{d_i}\circ \big(\mathcal{A}\overline{d_{i+1}}+\mathcal{A}R_{d_{i+1}}+R_{g_1}\circ (\overline{d_{i+1}}+R_{d_{i+1}})\big)-R_{d_i}\circ \mathcal{A}\overline{d_{i+1}}\big\|_{C^r}\\
  &+\big\|\overline{d_i}R_{g_1}\circ (\overline{d_{i+1}}+R_{d_{i+1}})-\overline{d_i}R_{g_1}\circ \overline{d_{i+1}}\big\|_{C^r}\\
  &+\big\|R_{g_1}\circ \overline{d_i}-R_{g_1}\circ (\overline{d_i}+R_{d_i})\big\|_{C^r}\\
  &\leq C_r\norm{R_{d_i},\mathcal{A}R_{d_{i+1}}+R_{g_1}\circ (\overline{d_{i+1}}+R_{d_{i+1}})}_{C^r}\\
  &\cdot\norm{R_{d_i},\mathcal{A}R_{d_{i+1}}+R_{g_1}\circ (\overline{d_{i+1}}+R_{d_{i+1}})}_{C^{r+1}}\\
  &+C_r\big\|R_{g_1},R_{d_{i+1}}\big\|_{C^r}\big\|R_{g_1},R_{d_{i+1}}\big\|_{C^{r+1}}\\
  &+C_r\big\|R_{g_1},R_{d_{i}}\big\|_{C^r}\big\|R_{g_1},R_{d_{i}}\big\|_{C^{r+1}}.
\end{align*}
Combined with
\begin{align*}
  &\big\|R_{g_1}\circ (\overline{d_{i+1}}+R_{d_{i+1}})-R_{g_1}\circ (\overline{d_{i+1}})\big\|_{C^r}\\
  &\leq C_r\big\|R_{g_1},R_{d_{i+1}}\big\|_{C^r}\big\|R_{g_1},R_{d_{i+1}}\big\|_{C^{r+1}},
\end{align*}
we get the conclusion immediately.
\end{proof}
\section{Proof of }

\end{document}